\documentclass[11pt]{amsart}

\usepackage{epigamath}

\usepackage[english]{babel}

\numberwithin{equation}{section}

\usepackage[shortlabels]{enumitem}

\usepackage{amsfonts}
\usepackage{amscd}
\usepackage{enumitem}
\usepackage{color}
\usepackage{comment}
\usepackage{epsfig}
\usepackage{amsopn}
\usepackage{xypic}
\usepackage{mathrsfs}
\usepackage{array}
\usepackage[usenames,dvipsnames]{pstricks}
\usepackage{epsfig}
\usepackage{pst-grad} % For gradients
\usepackage{pst-plot} % For axes
\usepackage[space]{grffile} % For spaces in paths
\usepackage{etoolbox} % For spaces in paths
\usepackage{mathtools}

\makeatletter % For spaces in paths
\patchcmd\Gread@eps{\@inputcheck#1 }{\@inputcheck"#1"\relax}{}{}
\makeatother

\usepackage{booktabs}
\usepackage{longtable}
\theoremstyle{plain}
\newtheorem{lemma}{Lemma}[section]
\newtheorem*{theorem*}{Theorem}
\newtheorem*{lemma*}{Lemma}
\newtheorem*{proposition*}{Proposition}
\newtheorem*{conjecture*}{Conjecture}
\newtheorem*{corollary*}{Corollary}
\newtheorem*{problem*}{Problem}
\newtheorem{theorem}[lemma]{Theorem}
\newtheorem{conjecture}[lemma]{Conjecture}
\newtheorem{corollary}[lemma]{Corollary}
\newtheorem{proposition}[lemma]{Proposition}

\theoremstyle{definition}
\newtheorem{definition}[lemma]{Definition}

\theoremstyle{remark}
\newtheorem{example}[lemma]{Example}
\newtheorem{remark}[lemma]{Remark}

\newcommand{\lra}{\longrightarrow}
\newcommand{\flra}[1]{\stackrel{#1}{\lra}}

\newcommand{\Z}{\mathbb{Z}}

\newcommand{\CC}{\mathbb{C}}
\newcommand{\QQ}{\mathbb{Q}}

\newcommand{\OO}{\mathcal{O}}
\newcommand{\te}{\otimes}
\newcommand{\sm}{\setminus}

\newcommand{\bv}{{\bf v}}

\renewcommand{\P}{\mathbb{P}}

\newcommand{\PP}{\mathbb{P}}

\DeclareMathOperator{\Bl}{Bl}

\DeclareMathOperator{\Hom}{Hom}

\DeclareMathOperator{\Pic}{Pic}

\DeclareMathOperator{\Ext}{Ext}
\DeclareMathOperator{\ext}{ext}

\DeclareMathOperator{\sHom}{\mathcal{H}\kern -.5pt\mathit{om}}
\DeclareMathOperator{\sTor}{\mathcal{T}\kern -1.5pt\mathit{or}}

\setlist[enumerate,1]{label={\rm(\arabic*)}, ref={\rm\arabic*}}

\EpigaVolumeYear{8}{2024} \EpigaArticleNr{7} \ReceivedOn{June 16, 2023}
\InFinalFormOn{December 1, 2023}
\AcceptedOn{December 19, 2023}

\title{Interpolation and moduli spaces of vector bundles on very general blowups of $\P^2$}
\titlemark{Interpolation and moduli spaces of vector bundles on very general blowups of $\P^2$}
\author{Izzet Coskun}
\address{Department of Mathematics, Statistics and CS, University of Illinois at Chicago, Chicago, IL 60607, USA}
\email{icoskun@uic.edu}

\author{Jack Huizenga}
\address{Department of Mathematics, The Pennsylvania State University, University Park, PA 16802, USA}
\email{huizenga@psu.edu}

\authormark{I.~Coskun and J.~Huizenga}

\AbstractInEnglish{In this paper, we study certain moduli spaces of vector bundles on the blowup of $\P^2$ in at least ten very general points. Moduli spaces of sheaves on general type surfaces may be nonreduced, reducible, and even disconnected. In contrast, moduli spaces of sheaves on minimal rational surfaces and certain del Pezzo surfaces are irreducible and smooth along the locus of stable bundles. We find examples of moduli spaces of vector bundles on more general blowups of $\P^2$ that are disconnected and have components of different dimensions. In fact, assuming the SHGH conjecture, we can find moduli spaces with arbitrarily many components of arbitrarily large dimension.}


\acknowledgement{During the preparation of this article the first author was partially supported by NSF FRG grant DMS 1664296 and  NSF grant DMS 2200684, 
and the second author was partially supported by NSF FRG grant DMS 1664303.}

\begin{document}

%%%%%%%%%%%%%%%%%%%%%%%%%%%%%%%
% Title page
%%%%%%%%%%%%%%%%%%%%%%%%%%%%%%%

\maketitle

\begin{prelims}

\DisplayAbstractInEnglish

\bigskip

\DisplayKeyWords

\medskip

\DisplayMSCclass

\end{prelims}

%%%%%%%%%%%%%%%%%%%%%
% Table of Contents
%%%%%%%%%%%%%%%%%%%%%

\newpage

\setcounter{tocdepth}{1}

\tableofcontents

%%%%%%%%%%%%%%%%%%%%%
% Content begins here
%%%%%%%%%%%%%%%%%%%%%

\section{Introduction}

In this paper, we study certain moduli spaces of vector bundles on the blowup of $\PP^2_\CC$ at $n \geq 10$ very general points. We describe their geometry very explicitly and find that they can have many connected components of different dimensions. In fact, assuming the Segre--Harbourne--Gimigliano--Hirschowitz (SHGH) conjecture (see Conjecture~\ref{conj-SHGH}), we find that they can have arbitrarily many components of arbitrarily large dimension. This is in strong contrast to the behavior of moduli spaces on minimal rational surfaces and certain del Pezzo surfaces. To the best of our knowledge,  our examples are the first time these phenomena have been observed on rational surfaces.

 Throughout the paper, let $X$ be the blowup of $\P^2$ at $n$ very general points $p_1,\ldots,p_n$.  Let $H$ denote the pullback of the class of a line.  Let $E_i$ denote the exceptional divisor lying over $p_i$, and set  $E =\sum_{i=1}^n E_i$.  The canonical divisor $K = K_X = -3H +E$ has self-intersection $K^2 = 9-n$.  The geometry of the surface $X$ changes dramatically based on the sign of $K^2$.  For $n\leq 8$, the divisor $-K$ is ample and $X$ is a del Pezzo surface.  For $n=9$, the divisor $-K$ is still effective.  On the other hand, for $n\geq 10$, the divisor $-K$ is not even effective. Our results will show that this dramatic change is also reflected in the behavior of moduli spaces of vector bundles on $X$. 
 
We polarize the surface $X$ by an ample divisor of the form $A_t = t H -E$. When $n \geq 10$, Nagata conjectures that $A_t$ is ample if $t > \sqrt{n}$; see \cite{Nagata}. Set $B= A_{\sqrt{n}} = \sqrt{n} H - E$ to be the conjectural nef ray. Nagata's conjecture is only known when $n$ is a perfect square, but the full Nagata conjecture is a consequence of the SHGH conjecture.  We record the numerical invariants of a vector bundle by $(r,c_1,\chi)$, where $r$ is the rank, $c_1$ is the first Chern class, and $\chi$ is the Euler characteristic.  We let $M_{X,A_t}(r,c_1,\chi)$ denote the moduli space of $A_t$-semistable sheaves with numerical invariants $(r,c_1,\chi)$.

In this paper, we study the moduli spaces $M_{X,A_t}(2,K,\chi)$.  We mostly focus on the case where $\chi\geq 1$ is positive, with special emphasis on the case where $\chi$ is the maximal Euler characteristic of an $A_t$-stable bundle.   For $10\leq n\leq 17$, this maximal Euler characteristic turns out to be $\chi = 2$.  On the one hand, the moduli spaces $M_{X,A_t}(2,K,\chi)$ are easiest to describe when $\chi$ is as large as possible; on the other hand, these spaces are also typically among the most pathological moduli spaces.  We show that when $n \leq 9$, the moduli spaces $M_{X, A}(2, K, \chi \geq 1)$ are empty for any ample divisor $A$ (see Proposition~\ref{prop-empty}). Similarly, if $n \geq 10$, then  $M_{X, A_t}(2, K, \chi \geq 1)$ is empty for $t> \frac{n}{3}$ (see Proposition~\ref{prop-empty10}). Hence, we will be interested in these moduli spaces when $n\geq 10$ and the polarization is close to the Nagata bound.

Our first main result shows that for any bundle $V$ with invariants $(2, K, \chi \geq 1)$, there is a unique effective divisor $D$ such that $V$ fits in an  exact sequence of the form
$$0 \lra \OO(D) \lra V \lra K(-D) \otimes I_Z \lra 0,$$
where $Z$ is a zero-dimensional scheme (see Theorem~\ref{thm-type}).  We will say that $V$ is a bundle of \emph{type} $D$.  The moduli spaces $M_{X,A_t}(2,K,\chi)$ are therefore stratified by the type of a bundle, and different types frequently give rise to different components in moduli spaces.  However, the possible types $D$ of stable bundles are extremely special:  if there is an $A_t$-stable bundle of type $D$, then $D$ must satisfy $2B\cdot D < B\cdot K$ and $\chi(D)\geq 1$ (see Proposition~\ref{prop-Dstablenecessary}). Given an effective $D$ with $\chi(D)\geq 1$ and $2B\cdot D < B\cdot K$, there is a unique value $t_D$ where  
$2 A_{t_D} \cdot D = A_{t_D} \cdot K$. A bundle $V$ of type $D$ can only be semistable for polarizations $A_t$ with $t \leq t_D$.  As~$t$ decreases past $t_D$ towards $\sqrt{n}$, the moduli spaces $M_{X,A_t}(2,K,\chi)$ can gain points parameterizing bundles of type $D$.    

When $n = 16$ or $25$, we are able to give the following description of moduli spaces that is unconditional on the Nagata or SHGH conjectures.  In these cases, there are only finitely many possible types $D$ of a semistable bundle.  Similar arguments could extend these descriptions to higher perfect squares $n$.

\begin{theorem}\label{thm11}\leavevmode
\begin{enumerate}
\item\label{thm11-1} Let $n=16$.
For $\frac{14}{3} < t < \frac{16}{3}$, the moduli space $M_{X, A_t}(2, K, 2)$ is isomorphic to $\PP^5$. For $4 < t < \frac{14}{3}$, the moduli space $M_{X, A_t}(2, K, 2)$ is isomorphic to a blowup of\, $\PP^5$ at $16$ points.

\item\label{thm11-2} Let $n=25$.  For $5 < t  \leq \frac{27}{5}$, the moduli space $M_{X, A_t}(2, K, 4)$ is isomorphic to a disjoint union of\, $25$ copies of\, $\PP^8$.
\end{enumerate}
\end{theorem}

In particular, part~\eqref{thm11-2} gives an explicit example of a reducible moduli space of vector bundles on a rational surface.

On the other hand, for $n$ which is not a perfect square, there are typically infinitely many effective divisors~$D$ satisfying $\chi(D)\geq 1$ and $2B\cdot D < B\cdot K$.  If we assume the Nagata conjecture, these divisors can be classified by solving a series of Pell's equations; we do this explicitly in Section~\ref{sec-D}.

\begin{example}
Suppose the Nagata conjecture holds.  If $n = 10$, then the effective divisors $D$ satisfying $\chi(D) \geq 1$ and $2B\cdot D < B\cdot K$ form an infinite list $$\OO, \quad 57H - 18E, \quad 2220H-702E, \quad 84357H-26676E, \quad \ldots.$$
These divisors can be read off from the continued fraction expansion of $\sqrt{10}$; see Theorem~\ref{thm-10points}.
\end{example}

If we assume the SHGH conjecture, then the types $D$ which could contribute to the moduli spaces have good cohomological properties which makes it possible to completely describe the moduli spaces.  For example, if $10\leq n\leq 16$, then $D$ is the class of a reduced, irreducible, rigid curve on $X$ (see Theorem~\ref{thm-irr}).  The next theorem then summarizes our results from Section~\ref{sec-components} which describe the structure of the moduli spaces.

\begin{theorem}\label{thm-intro}
Let $10 \leq n \leq 15$, and assume the SHGH conjecture.
\begin{enumerate}
\item If\, $t > \frac{n}{3}$, then $M_{X, A_t}(2, K, 2)$ is empty.
\item Suppose $11\leq n\leq 15$. As $t$ decreases past $\frac{n}{3}$, $M_{X, A_t}(2, K, 2)$ acquires a component isomorphic to $\PP^{n-11}$. For $11\leq n \leq 12$, this component persists without modification as $t$ decreases to  $\sqrt{n}$. For $13 \leq n \leq 15$, this component is blown up at $n$ points as $t$ decreases past $\frac{n-2}{3}$ and then persists without modification as $t$ decreases to $\sqrt{n}$.
\item For every nontrivial, nonexceptional divisor $D$ satisfying $\chi(D)\geq 1$ and $2B \cdot D < B \cdot K$, $M_{X, A_t}(2, K, 2)$ acquires a new component isomorphic to $\PP^{-\chi(2D-K) -1}$ as  $t$ decreases past $t_D$. This component persists without modification as $t$ decreases to $\sqrt{n}$.
\item This is a complete description of the components of\, $M_{X, A_t}(2, K, 2)$, and they are all disjoint.
\end{enumerate}
\end{theorem}

We list the first several components of each of the moduli spaces when $10\leq n\leq 13$ in tables in Examples~\ref{ex-components10} and~\ref{ex-components13}.  A detailed study of the possible divisors $D$ shows that each of the moduli spaces in Theorem~\ref{thm-intro} will have arbitrarily many components of arbitrarily large dimension if $t$ is sufficiently close to $\sqrt{n}$.  This observation implies the following corollary that applies to moduli spaces with arbitrary Euler characteristic.

 \begin{corollary}
 Assume the SHGH conjecture, and let $10 \leq n \leq 12$. Let $\chi \leq 2$ be an integer, and let $k$ and $r$ be positive integers. There exists an $\epsilon >0$ such that if $\sqrt{n} < t < \sqrt{n}+ \epsilon$, then the moduli space $M_{X, A_t}(2, K , \chi)$ has at least $k$ irreducible components of dimension $r$.
 \end{corollary}

When the polarization is fixed, moduli spaces of sheaves on surfaces behave well as $\chi$ tends to negative infinity. For example, by a theorem of O'Grady, see \cite{OGrady}, the moduli spaces are irreducible, reduced, and normal. However, for arbitrary $\chi$, the moduli spaces can be poorly behaved. For example, moduli spaces of sheaves on general type or elliptic surfaces can be reducible, nonreduced, and even disconnected (see \cite{CoskunHuizengaPathologies, CoskunHuizengaKopperDisconnected, Friedman, Friedman2,Kotschick, Mestrano, MestranoSimpson, OkonekVandeven} for some examples).

Let $Y$ be a birationally ruled surface, and let $F$ be the class of the fiber. Let $A$ be a polarization such that $(K_Y  + F) \cdot A < 0$. Walter proves that  the moduli space $M_{Y, A}({\bf v})$ is then irreducible provided that it is nonempty; see \cite{Walter}. In particular, all nonempty moduli spaces of sheaves  on $\PP^2$, Hirzebruch surfaces and $X$ with $n \leq 6$ are irreducible for every polarization; these moduli spaces have been studied in detail (see for example \cite{LePotier, CoskunHuizengaHExist, LevineZhang}). Similarly, for any rational surface $Y$, there exist polarizations $A$ satisfying $(K_Y + F) \cdot A < 0$. For example, this is the case on $X$ for $A_t$ with $t \gg 0$. The nonempty moduli spaces are irreducible on $X$ for such polarizations and have been studied in  \cite{Zhao}. In contrast, our results show that the irreducibility may fail when Walter's condition is violated. 

Our results are in part inspired by questions concerning the topology of moduli spaces. G\"{o}ttsche, see \cite{Gottsche}, computed the Betti numbers of the Hilbert schemes $Y^{[n]}$ of $n$-points on a smooth projective surface $Y$  and observed that they stabilize as $n$ tends to infinity. In fact, the Betti numbers monotonically increase as $n$ increases. Coskun and Woolf, see \cite{CoskunWoolf}, conjectured that the Betti numbers stabilize for moduli spaces of sheaves in general as $\chi$ tends to negative infinity and that the stable Betti numbers are  independent of the rank and the polarization. They proved the conjecture for moduli spaces on rational surfaces when the polarization satisfies $K_Y \cdot A < 0$ and the moduli space does not contain any strictly semistable sheaves. Our examples show that in the absence of the assumption $K_Y \cdot A < 0$, the topology of the moduli spaces can be fairly complicated. In particular, even on rational surfaces, the Betti numbers of moduli spaces are not monotonically increasing as $\chi$ decreases. Examples of this phenomenon were previously known on certain elliptic and general type surfaces; see \cite{CoskunHuizengaKopperDisconnected, Kotschick, OkonekVandeven}.  

\subsection*{Organization of the paper}
In Section~\ref{sec-prelim}, we recall the Nagata and SHGH conjectures and collect basic facts concerning very general blowups of $\PP^2$. In Section~\ref{sec-type}, we define the type of a bundle $V$ with character $(2, K , \chi \geq 1)$ and show that it is unique. In Section~\ref{sec-D}, we study effective divisors  $D$ that satisfy $\chi(D)\geq 1$ and $2B \cdot D < B \cdot K$ and explain how to classify them. In Section~\ref{sec-cohomology}, we study the cohomology of  such $D$ and associated divisors which are relevant to the calculation of the tangent space of the moduli space. In Section~\ref{sec-components}, we classify the components of the moduli spaces and prove our main theorems.  Finally, in Section~\ref{sec-square}, we study the cases where $n$ is a perfect square, where we can make our results independent of the SHGH conjecture.

\subsection*{Acknowledgments}
We would like to thank Benjamin Gould, Daniel Huybrechts, John Kopper, Daniel Levine, Yeqin Liu, Dmitrii Pedchenko, Matthew Woolf, and Junyan Zhao for valuable conversations about moduli spaces of sheaves on surfaces.  We would also like to thank the referee for their detailed comments on the paper.

\section{Preliminaries}\label{sec-prelim}

\subsection{Notation}
Throughout the paper, we work over the field $\CC$ of complex numbers.   Let $X$ be the blowup of $\P^2$ at $n$ very general points $p_1,\ldots,p_n$.  The Picard group of $X$ is
$$\Pic X \cong \Z H \oplus \Z E_1\oplus \cdots \oplus \Z E_n,$$
where $H$ is the pullback of a line in $\P^2$ and $E_1,\ldots,E_n$ are the exceptional divisors.  We have $H^2 = 1$, $H\cdot E_i = 0$, $E_i^2 = -1$, and $E_i\cdot E_j = 0$ for $i\neq j$.  We set $E = \sum_{i=1}^n E_i$.  For brevity, we let $E_{i_1\dots i_k} := \sum_{j=1}^k E_{i_j}$. For  example, $E_{123} = E_1 + E_2 + E_3$.  We write $\OO = \OO_X$ and $K = K_X$ for the trivial bundle and canonical bundle, respectively, and note that
$$ K = - 3H + E.$$
We compute $K^2 = 9-n$.  

\subsection{Ample divisors}
In this paper, we will study polarizations of $X$ of the form $A_t = tH - E$, where $t$ is a real number.  Since $A_t^2 = t^2 - n$ and $A_t\cdot H = t$, if $A_t$ is ample, then  $t> \sqrt{n}$.  The famous conjecture of Nagata claims that the converse is true once $n\geq 10$.

\begin{conjecture}[Nagata, \textit{cf.} \cite{Nagata}]
Let $n\geq 10$.  If\, $t> \sqrt{n}$, then $A_t$ is ample.  In particular,
$$B := A_{\sqrt{n}} = \sqrt{n} H - E$$
is nef.
\end{conjecture}

Nagata shows the conjecture is true when $n$ is a perfect square.  For other $n$, partial results towards the Nagata conjecture can be proved by exhibiting ample divisors $A_t$ with $t$ as close to $\sqrt{n}$ as possible. For $\alpha\geq \sqrt{n}$, we will call the statement that $A_{\alpha}$ is nef the $\alpha$-\emph{Nagata conjecture}.

\subsection{Linear series and the SHGH conjecture}
Consider a divisor class $D = dH - \sum_i m_i E_i$ on $X$ with $d\geq 0$.  In general, it is a highly nontrivial problem to compute the dimension of the linear series $|D|$ or, equivalently, the cohomology of the line bundle $\OO(D)$.  The  Segre--Harbourne--Gimigigliano--Hirschowitz (SHGH) conjecture provides an algorithm to compute this dimension.  We will call $D$ \emph{special} if both $h^0(\OO(D))$ and $h^1(\OO(D))$ are nonzero. Otherwise, $D$ is \emph{nonspecial}. If $D$ is nonspecial, then the cohomology of $\OO(D)$ is easily determined by the Euler characteristic $\chi(\OO(D))$.  Recall that a $(-1)$-curve on $X$ is a smooth rational curve $C\subset X$ with $C^2 = -1$.

\begin{conjecture}[SHGH, \textit{cf.} \cite{Segre,Harbourne,Gimigliano,Hirschowitz}]\label{conj-SHGH}
The divisor $D$ is special if and only if it contains a multiple $(-1)$-curve in its base locus.
\end{conjecture}

If $n\leq 9$, then the SHGH conjecture is true (see \textit{e.g.} \cite{CilibertoMiranda}), so the conjecture becomes most interesting for $n\geq 10$.

\begin{remark} The following consequences of the conjecture are frequently useful:
\begin{enumerate}\label{rem-SHGHConsequences}
\item If $D$ is a reduced curve on $X$, then $D$ is nonspecial and $\chi(D) \geq 1$.

\item If $D$ is a reduced and irreducible curve on $X$ with $D^2 < 0$, then $D$ is a $(-1)$-curve.  Indeed, by Riemann--Roch, a large multiple $kD$ has $\chi(kD) < 0$, but $kD$ is effective.  Therefore, $kD$ is special, and the only possibility is that $D$ is a $(-1)$-curve.

\item\label{rem-SHGHC-3} Suppose $D = dH - mE$ is a \emph{homogeneous} divisor class.  If $n\geq 10$, then $D$ is nonspecial.

\item\label{rem-SHGHC-4} The SHGH conjecture implies the Nagata conjecture.  For suppose that $t > \sqrt{n}$ and $A_t$ is not ample, so that by the Nakai--Moishezon criterion, there is an irreducible curve class $C = dH - \sum_i m_i E_i$ with $C \cdot A_t < 0$.  Then $C$ is nonspecial, and if we permute the exceptional divisors, we get additional nonspecial classes.  Summing over the symmetric group, we can obtain an effective homogeneous divisor class $D = d'H - m'E$ with $D\cdot A_t < 0$.  This implies $\frac{d}{m} < \sqrt{n}$ and $D^2 < 0$.  Then large multiples $kD$ have $\chi(kD)< 0$ and they are effective, contradicting that they are nonspecial by~\eqref{rem-SHGHC-3}. 
\end{enumerate}
\end{remark}

Since the full SHGH conjecture is quite challenging, it is useful to have results which make partial progress towards the SHGH conjecture.  Here there are two main flavors of result: either one can bound the multiplicities $m_i$ (see \textit{e.g.} \cite{DumnickiJarnicki} and \cite{Yang}), or one can focus on homogeneous series.  

Studying homogeneous series essentially boils down to two infinite families of statements.  Let $D = dH - mE$ be a homogeneous series.  For $\alpha\geq  \sqrt{n}$, we say that $\alpha$-\emph{nonspeciality} holds if whenever $\frac{d}{m} \geq \alpha$, $D$ is nonspecial.  On the other hand, for $\beta \leq \sqrt{n}$, we say that $\beta$-\emph{emptiness} holds if whenever $\frac{d}{m} \leq \beta$,  $D$ is noneffective.  Note that if $\beta \leq \sqrt{n}$, then $A_\beta \cdot A_{n/\beta} = 0$, so $\beta$-emptiness implies $A_{n/\beta}$ is nef by the same argument as in Remark~\ref{rem-SHGHConsequences}\eqref{rem-SHGHC-4}.  Thus $\beta$-emptiness implies the $\frac{n}{\beta}$-Nagata conjecture.  Various instances of these statements are theorems in the literature; see for example \cite{Petrakiev,CilibertoMiranda}.  Many of the strongest statements have been proved in the first case $n=10$.

\begin{example}
For $n=10$, we have $\frac{2280}{721}$-emptiness, see \cite{Petrakiev}, and $\frac{174}{55}$-nonspeciality, see \cite{CilibertoMiranda}, and the $\frac{721}{228}$-Nagata conjecture holds.
\end{example}

\subsection{Moduli spaces of vector bundles}\label{sec-prelimModuli}
Let $A$ be  an ample divisor on $X$. Let $V$ be a torsion-free sheaf on $X$ with Chern character ${\bf v}$. In this paper, it will be convenient to record ${\bf v}= (r, c_1, \chi)$ by the rank $r$, the first Chern class $c_1(V)$, and the Euler characteristic $\chi(V)$. The {\em $A$-slope} $\mu_A(V)$, the {\em Hilbert polynomial} $P_{A,V}(m)$, and the {\em reduced Hilbert polynomial} $p_{A,V}(m)$ are defined by 
$$\mu_A(V) = \frac{c_1(V) \cdot A}{r}, \quad P_{A,V}(m)= \chi(V (mA)), \quad p_{A,V}(m) = \frac{P_{A,V}(m)}{r},$$
respectively. The sheaf $V$ is {\em $\mu_A$}-{\em stable} (respectively, {\em $\mu_A$}-{\em semistable}\/) if for all proper subsheaves $W \subset V$, $\mu_A(W) <\, \mu_A(V)$ (respectively, $\mu_A(W) \leq \mu_A(V)$). The sheaf $V$ is $A$-{\em stable} (respectively, $A$-{\em semistable}\/) if for all proper subsheaves $W \subset V$, $p_{A,W}(m) <\, p_{A,V}(m)$ (respectively, $p_{A,W}(m) \leq p_{A,V}(m)$)
for $m \gg 0$.  Gieseker, see \cite{Gieseker}, and Maruyama, see \cite{Maruyama},  constructed projective moduli spaces $M_{X, A}({\bf v})$ parameterizing $A$-semistable sheaves. We refer the reader to \cite{HuybrechtsLehn} and \cite{LePotier} for the properties of these moduli spaces.

\section{Types of bundles with positive Euler characteristic}\label{sec-type}

\subsection{Types of bundles} 
Throughout this section, we let $\bv$ be the Chern character $\bv = (r,c_1,\chi) = (2,K,\chi)$, where $\chi\geq 1$ is a positive integer. 
The first main result in the paper shows that the positivity assumption on $\chi$ allows us to neatly classify vector bundles of character $\bv$ into various types.  These will give rise to distinct components in moduli spaces.

\begin{definition}
  Let $\bv = (2,K,\chi)$ with $\chi\geq 1$, and let $D\in \Pic(X)$ be a (possibly trivial) \emph{effective} divisor class on $X$ satisfying $\chi(D) \geq 1$.  A vector bundle $V$ of character $\bv$ has \emph{type $D$} if it fits in an exact sequence of the form
  $$0\lra \OO(D) \lra V\lra K(-D)\te I_Z\lra 0$$
  for some zero-dimensional scheme $Z$ of length $2\chi(\OO(D))-\chi$.
\end{definition}

Observe that $\chi(K(-D)) = \chi(\OO(D))$ and $\chi(K(-D)\te I_Z)= \chi(\OO(D)) - l(Z)$, so the assumption on the length of $Z$ is necessary to give $\chi(V) = \chi$.  The divisor $D$ must also have $\chi(D) \geq 1$ in order for $\chi \geq 1$ to be possible.  We first show that the type exists and is unique.

\begin{theorem}\label{thm-type}
Let $\bv = (2,K,\chi)$ with $\chi\geq 1$.  Any vector bundle $V$ of character $\bv$ is of type $D$ for exactly one  effective divisor class $D$.
\end{theorem}

\begin{proof}
First we show that a type exists.  Since $\chi(V) \geq 1$, at least one of $h^0(V)$ or $h^2(V)$ is nonzero.  

Suppose $h^2(V) \neq 0$.  Then $h^0(V^* \te K)\neq 0$, so $\Hom(V,K)\neq 0$.  Pick a nonzero map $V\to K$, and let $F \subset K$ be its image.  Then $F$ is of the form $K(-D)\te I_Z$ for an effective divisor $D$ and a zero-dimensional scheme $Z\subset X$.  Consider the kernel
$$0\lra G \lra V \lra K(-D) \te I_Z\lra 0.$$
Basic facts about homological dimension and the Auslander--Buchsbaum formula imply that $G$ is locally free since it is the kernel of a surjective mapping from a vector bundle to a torsion-free sheaf on a smooth surface (see \cite[Section~1.1, p.~4]{HuybrechtsLehn}). 
By Chern class considerations, we deduce $G \cong \OO(D)$, and $V$ has type $D$.

If instead $h^0(V) \neq 0$, we reduce to the previous case.  Since $H^0(V) \cong \Hom(\OO,V)$, we pick a nonzero (hence injective) homomorphism $\OO\to V$ and consider its cokernel
$$0\lra \OO \lra V \lra F \lra 0.$$
Let $T$ be the torsion subsheaf of $F$, so we have an exact sequence
$$0\lra T\lra F\lra G \lra 0.$$
The first Chern class of $T$ is a positive $\Z$-linear combination of any curves in the support of $T$, so it is a (possibly empty) effective divisor $D$.  Then  $G$ is a rank $1$ torsion-free sheaf with $c_1(G) = K-D$, so it is of the form
$$G = K(-D)\te I_Z$$
for a zero-dimensional scheme $Z$.  Then $h^2(G)\neq 0$,  so $h^2(F)\neq 0$. Hence,  $h^2(V) \neq 0$, and we are reduced to the previous case.

For the uniqueness, suppose $V$ has type $D$ and type $D'$.  Twisting the type $D$ exact sequence by $-D$ shows that $V(-D)$ has a section.  But twisting the type $D'$ exact sequence by $-D$ gives
$$0\lra \OO(D'-D) \lra V(-D) \lra K(-D-D')\te I_Z \lra 0.$$
Since $H\cdot (K-D-D') < 0$, the divisor $K-D-D'$ is not effective.  Therefore, $D'-D$ is effective.  By a symmetric argument, $D-D'$ is effective.  These two facts are only compatible if $D = D'$.
\end{proof}

In particular, we have the following corollary.  

\begin{corollary}\label{cor-section}
Let $V$ be a vector bundle of rank $2$ with $c_1(V) = K$ and $\chi(V) \geq 1$.  Then $h^0(V)$ and $h^2(V)$ are both nonzero.
\end{corollary}

\begin{proof}
The bundle $V$ has type $D$ for some $D$, and from the defining sequence, we see that $V$ has the required cohomology.
\end{proof}

\begin{remark}
In the special case where $\chi = 1$, all three cohomology groups $H^0(V)$, $H^1(V)$, and $H^2(V)$ must be nonzero.  

For example, let us discuss what happens when $n=10$ and $\chi = 1$.  The type $\OO$ bundles fitting into sequences of the form
$$0\lra \OO \lra V\lra K\te I_p\lra 0$$
have $h^0(V) = h^1(V) = h^2(V) = 1$.  Additionally, if $A$ is any ample divisor and $U \subset M_A(\bv)$ is any component whose general member is a vector bundle, then every sheaf in that component must have nonvanishing cohomology in every degree.  This exhibits a strong failure of the ``weak Brill--Noether'' property for these spaces, in stark contrast with known results for minimal rational surfaces and del Pezzo surfaces (see for example \cite{CoskunHuizengaWBN,CoskunHuizengaBN,LevineZhang}).
\end{remark}

The type of a bundle $V$ can be determined cohomologically. 

\begin{corollary}\label{cor-cohomologicalType}
  Let $V$ be a bundle of character $\bv= (2,K,\chi)$ with $\chi \geq 1$. Partially order $\Pic(X)$ by the relation $D'\leq D$ if\, $D - D'$ is effective.  Then the type of\, $V$ is the unique maximal element in
  $$\{D'\in \Pic(X) : D' \textrm{ effective and } h^0(V(-D')) \neq 0\}.$$
\end{corollary}

\begin{proof}
  Suppose $V$ has type $D$.  By definition, $D$ is in the set. Let $D'$ be any  effective divisor with $h^0(V(-D'))\neq 0$.  Twisting the type $D$ sequence by $-D'$ gives
  $$0\lra \OO(D-D') \lra V(-D') \lra K(-D-D')\te I_Z\lra 0.$$
Since $h^0(V(-D'))\neq 0$, we must have that $D-D'$ is effective, so $D'\leq D$.
\end{proof}

This cohomological definition of type restricts the ways in which bundles of one type can specialize to another.  

\begin{corollary}
Let $\bv = (2,K,\chi)$ with $\chi\geq 1$.  Suppose $V_s/S$ is a flat family of vector bundles on $X$ of character $\bv$, parameterized by an irreducible base $S$, and that $V_s$ has type $D$ for a general $s\in S$.  If $s'\in S$ is such that $V_{s'}$ has some type $D'$, then $D\leq D'$.
\end{corollary}

\begin{proof}
For a general $s\in S$, we have $h^0(V_s(-D)) \neq 0$.  By semicontinuity, we get $h^0(V_{s'}(-D)) \neq 0$.  Then Corollary~\ref{cor-cohomologicalType} gives $D \leq D'$.
\end{proof}

\subsection{Preliminary results on stability}
The next result shows that once we are concerned with stability, the spaces $M_A(\bv)$ are not interesting until $10$ or more points are blown up.

\begin{proposition}\label{prop-empty}
Let $\bv = (2,K,\chi)$ with $\chi \geq 1$.  If $n\leq 9$, then the moduli space $M_A(\bv)$ is empty for \emph{every} ample divisor $A$.
\end{proposition}

\begin{proof}
Suppose $V\in M_A(\bv)$ is an $A$-semistable sheaf.  Since $\chi\geq 1$, we have either $h^0(V) > 0$ and $\Hom(\OO,V) \neq 0$, or $h^2(V) > 0$ and $\Hom(V,K) \neq 0$.  Notice that because $n\leq 9$, we have $A\cdot K < 0$ since $-K$ is effective.  We have $\mu_A(V) = \frac{1}{2}A\cdot K$, so $\mu_A(K) < \mu_A(V) < \mu_A(\OO)$.  Then either a map $\OO\to V$ or a map $V\to K$ destabilizes $V$.
\end{proof}

On the other hand, for $n\geq 10$, we focus on polarizations of the form $A_t = tH - E$.  Here we find that the spaces $M_{A_t}(\bv)$ are empty until $A_{t}$ becomes sufficiently close to $B= A_{\sqrt{n}} = \sqrt{n}H - E$.

\begin{proposition}\label{prop-empty10}
Let $\bv = (2,K,\chi)$ with $\chi\geq 1$.  If $n\geq 10$, then the moduli space $M_{A_t}(\bv)$ is empty when $t > \frac{n}{3}$.
\end{proposition}

\begin{proof}
When $t>\frac{n}{3}$, we have $A_t \cdot K < 0$, and we proceed as in the previous proof.
\end{proof}

Next we investigate the stability of bundles of type $D$.  The existence of a stable bundle of type $D$ imposes strong restrictions on $D$.

\begin{proposition}\label{prop-Dstablenecessary}
  Suppose $V$ is a bundle of character $\bv$ and type $D$, and assume there  is a polarization $A_{t_0}$ such that $V$ is $\mu_{A_{t_0}}$-semistable. Then we must have
  $$2B\cdot D < B\cdot K,$$
  and there is a unique polarization $A_{t_D}$ such that $\OO(D)$ and $K(-D)$ have the same slope.  It satisfies
  $$2A_{t_D}\cdot D = A_{t_D}\cdot K.$$
\end{proposition}

\begin{proof}
Suppose $V$ is $\mu_{A_{t_0}}$-semistable and that it fits in an exact sequence
$$0\lra \OO(D) \lra V\lra K(-D) \te I_Z\lra 0.$$
We have $t_0 \leq \frac{n}{3}$ by Proposition~\ref{prop-empty10}, and we must have
$$A_{t_0}\cdot D \leq A_{t_0}\cdot(K-D),$$
or
$$2 A_{t_0}\cdot D \leq A_{t_0}\cdot K.$$
Now for variable $t$, consider the relationship between $2A_t\cdot D$ and $A_t\cdot K$.  Both quantities vary linearly in~$t$.  For $t = t_0$, we have $2A_t\cdot D \leq A_t\cdot K$, and for $t> \frac{n}{3}$, we find $2A_t \cdot D >  A_t \cdot K$.  Thus there is a unique $t_D$ between $t_0$ and $\frac{n}{3}$ such that $2A_{t_D} \cdot D = A_{t_D}\cdot K$.  Furthermore, if we take $t = \sqrt{n}$, we get $2B\cdot D < B\cdot K$.
\end{proof}

Due to the proposition, we study curves $D$ satisfying the inequality $2B\cdot D < B\cdot K$ in more detail in the next section.

\section{Effective divisors \texorpdfstring{$\boldsymbol{D}$}{$D$} satisfying \texorpdfstring{$\boldsymbol{2B\cdot D < B\cdot K}$}{$2B\cdot D < B\cdot K$}}\label{sec-D}

In this section, we classify the effective divisors $D$ satisfying $2B\cdot D< B\cdot K$ and $\chi(D) \geq 1$, at least when $n$ is small.  For a complete answer to this question, we will need the Nagata conjecture to know that $B$ is nef. However, with a little more care, we can avoid using the Nagata conjecture and instead classify divisors satisfying $2A_t \cdot D \leq A_t \cdot K$ whenever $A_t$ is known to be an ample divisor.  Notice that the inequality $2 A_t\cdot D \leq A_t\cdot K$ implies $2 B\cdot D < B\cdot K$, as in the proof of Proposition~\ref{prop-Dstablenecessary}.  Conversely, if the Nagata conjecture holds, then $2B\cdot D<B\cdot K$ implies that $2A_t\cdot D < A_t \cdot K$ for $t$ slightly greater than $\sqrt{n}$.  Also, since $A_{n/3} \cdot K = 0$, the polarizations we are interested in all have $t \leq \frac{n}{3}$.

\subsection{General restrictions on $\boldsymbol{D}$}
In this subsection, we prove several preliminary results which restrict the possibilities for a divisor $D$.  

\begin{proposition}\label{prop-chiUpper}
  Suppose $D$ is an effective divisor and $A_t$ is an ample divisor with
  $$2A_t\cdot D\leq A_t\cdot K.$$
  Then
  $$\chi(D) < \frac{n-1}{8}.$$
  In particular, if\, $10\leq n\leq 17$ and $\chi(D) \geq 1$, then $\chi(D) = 1$.
\end{proposition}

\begin{proof}
  Since $2 A_{n/3} \cdot D > 0  = A_{n/3} \cdot K$, the assumption $2A_t \cdot D \leq A_t\cdot K$ implies there is an ample divisor $A = A_{t_D}$ with $2 A\cdot D = A \cdot K$.  Then $A\cdot (2D - K) = 0$, so by the Hodge index theorem, we have $(2D-K)^2 < 0$. Expanding and rearranging, we get $4D\cdot(D-K) < - K^2$, and so
  $$\frac{1}{2} D\cdot (D-K) < \frac{n-9}{8}.$$
  The required inequality follows from the Riemann--Roch formula $\chi(D) = 1+\frac{1}{2}D\cdot(D-K)$.
\end{proof}

\begin{definition}\label{def-balanced}
Let $D = dH - \sum_i m_iE_i$ be a divisor.  We say that the multiplicities are \emph{balanced} if $|m_i- m_j|\leq 1$ for all $i,j$.

On the other hand, if $D$ is \emph{not} balanced, we construct a sequence $D= D_0,\ldots,D_k$ of divisors by iteratively increasing one of the smallest multiplicities by $1$ and decreasing one of the largest multiplicities by $1$, stopping when we arrive at a balanced divisor $D_k$.  We say that $D$ is \emph{$k$ steps away from having balanced multiplicities}.  The number $k$ is independent of any choices made in this construction.
\end{definition}

\begin{lemma}
  Suppose $D_0 = dH-\sum_i m_i E_i$ is any divisor and that $D_0$ is $k$ steps away from having balanced multiplicities.  Let $D_k$ be the balanced divisor obtained from $D_0$ as in Definition~\ref{def-balanced}.  Then
  $$\chi(D_k) \geq\chi(D_0)+k.$$
\end{lemma}

\begin{proof}
  We claim that each step of ``rebalancing'' the multiplicities increases the Euler characteristic by at least $1$.  Without loss of generality, suppose that $D_0$ is not balanced, $m_1-m_2 \geq 2$, and $D_1 = D_0+E_1-E_2$.  Write $F = E_1-E_2$.  Then since $D_0\cdot(D_0-K) = 2\chi(D_0)-2$, we have
  $$D_1\cdot (D_1-K) = D_0\cdot(D_0-K) +2F\cdot D_0 - F\cdot K + F^2=2\chi(D_0)-2+2m_1-2m_2-2\geq2 \chi(D_0),$$
  so $\chi(D_1) \geq \chi(D_0)+1$.  Repeating proves the result.
\end{proof}

\begin{proposition}\label{prop-balanced}
  Suppose $D = dH - \sum_i m_iE_i$ is an effective divisor with $\chi(D)\geq 1$ and $A_t$ is an ample divisor with
  $$2A_t\cdot D \leq A_t \cdot K.$$
  Let $\chi_{\max}$ be the maximal Euler characteristic among effective divisors $D$ satisfying this inequality.  If we put $\ell = \chi_{\max} - \chi(D)$, then $D$ is at most $\ell$ steps away from having balanced multiplicities.  

In particular, if we additionally assume $10\leq n \leq 17$, then $\chi(D) = \chi_{\max}=1$, and so $D$ is balanced. 
\end{proposition}

\begin{proof}
Consider the divisor $D_k$ which rebalances $D$.  Since $\chi(D_k) \geq \chi(D) + k$ and $\chi(D) \geq 1$, we see that $D_k$ is effective.  Also, $D_k\cdot A_t = D_0\cdot A_t$, so we have $2A_t\cdot D_k \leq A_t\cdot K$.  By the definition of $\chi_{\max}$, we find $\chi(D)+k\leq \chi(D_k)\leq \chi_{\max}$, so $k \leq \chi_{\max}-\chi(D)$. 
\end{proof}

Next we quantify how far away a balanced divisor is from having equal multiplicities.

\begin{definition}
Let $D = dH - \sum_i m_iE_i$ be a balanced divisor, so that $|m_i-m_j|\leq 1$ for all $i,j$.  Then there are $k$ copies of some multiplicity $m+1$ and $n-k$ copies of multiplicity $m$.  We let $\ell = \min\{k,n-k\}$, and we say that $D$ is \emph{$\ell$ steps away from having  equal multiplicities}.  
\end{definition}

\begin{proposition}\label{prop-equal}
  Suppose $D$ is a balanced effective divisor with $\chi(D)\geq 1$ and $A_t$ is an ample divisor with
  $$2A_t\cdot D \leq A_t \cdot K.$$
  Suppose that $D$ is $\ell$ steps away from having equal multiplicities.  Then
  $$\ell < \frac{1}{2}\left(n-\sqrt{(8\chi(D)+1)n}\right).$$
  In particular, if\, $10\leq n\leq 12$, then we get $\chi(D) = 1$, $\ell = 0$, and $D$ has equal multiplicities.
\end{proposition}

\begin{proof}
  Suppose there are $1\leq k \leq n-1$ copies of multiplicity $m+1$ and $(n-k)$ copies of multiplicity $m$, so that the average multiplicity is $m + \frac{k}{n}$.  Let $E'$ be the sum of the $k$ exceptional divisors with multiplicity $m+1$, and let $E''$ be the sum of the remaining $n-k$ exceptional divisors, so that $D = dH - (m+1)E'-mE''$.    Put
  $$G=\left(1-\frac{k}{n}\right)E'-\frac{k}{n}E'',$$
  so that $D':=D+G = dH-(m+\frac{k}{n})E$ is a $\QQ$-divisor with equal multiplicities.  Observe that the sum of the multiplicities in $G$ is zero, so $G$ intersects any divisor with equal multiplicities in $0$.  In particular, as in the proof of Proposition~\ref{prop-chiUpper}, the Hodge index theorem still gives $\frac{1}{2}D'\cdot (D'-K) < \frac{n-9}{8}$.   Since $D\cdot (D-K)=2\chi(D)-2$, we compute
  $$D'\cdot (D'-K) = 2\chi(D)-2+ 2G\cdot D-G\cdot K+G^2 = 2\chi(D)-2+ \frac{k(n-k)}{n}.$$
  But $k(n-k)=\ell(n-\ell)$, so we must have the inequality
  $$\chi(D)-1+\frac{\ell(n-\ell)}{2n} = \frac{1}{2}D'\cdot(D'-K)<\frac{n-9}{8}.$$
  Solving the inequality for $\ell$ and recalling that by definition $\ell \leq n/2$, we get the required upper bound on~$\ell$.
\end{proof}

In fact, the results in this section have the following partial converse, which will help us to find all the divisors $D$.

\begin{proposition}\label{prop-ineq}
  Suppose $D=dH-(m+1)E'-mE''$ is a balanced effective divisor with $1\leq \chi(D) < \frac{n-1}{8}$.  Assume $m\geq 0$  and that $D$ is $\ell$ steps away from having equal multiplicities, where
  $$\ell < \frac{1}{2}\left(n-\sqrt{(8\chi(D)+1)n}\right).$$ Then $$2B\cdot D < B\cdot K.$$
\end{proposition}

\begin{proof}
  First suppose that $D' = dH-mE$ is a $\QQ$-divisor with $d\geq 0$, $m\geq 0$, and $(2D'-K)^2<0$. This inequality is
  $$(2d+3)^2- n(2m+1)^2<0,$$
  which factors as a difference of squares
  $$((2d+3)-\sqrt n (2m+1))((2d+3)+\sqrt{n}(2m+1))<0.$$
The first factor in this product is just $\frac{1}{\sqrt{n}} B\cdot(2D'-K)$.  The second factor is positive since $d\geq 0$ and $m\geq 0$, so we conclude that $B\cdot (2D'-K) < 0$ and $2B\cdot D' < B\cdot K$.

Now letting $D$ be the divisor in the statement of the proposition, we form a balanced $\QQ$-divisor $D'$ with the same average multiplicity as in the proof of Proposition~\ref{prop-equal}.  As in the proof of Proposition~\ref{prop-equal}, we find that
$$\frac{1}{8}(2D'-K)^2 = \frac{1}{2}D'\cdot(D'-K) +\frac{K^2}{8} = \chi(D)-1+\frac{\ell(n-\ell)}{2n}+\frac{9-n}{8}.$$
Our upper bound on $\ell$ then implies that $(2D'-K)^2 < 0$, and by the argument in the first paragraph, $2B\cdot D' < B\cdot K$.  As $B\cdot D' = B\cdot D$, we are done.
\end{proof}

\subsection{Ten through twelve points}
When $10\leq n\leq 12$, it is now a matter of number theory to describe all the possible divisors $D$.

\begin{theorem}\label{thm-10points}
  For $10\leq n\leq 12$, we consider the sequence of convergents
  $$\frac{p_1}{q_1},\,\frac{p_2}{q_2},\,\frac{p_3}{q_3},\,\frac{p_4}{q_4},\,\ldots$$
  of the continued fraction expansion of $\sqrt{n}$.  Here the odd convergents are less than $\sqrt{n}$, and the even convergents are greater than $\sqrt{n}$.
\begin{enumerate}
\item\label{thm-10p-1} For any positive odd integer $k$, we let $d_k = \frac{1}{2}(p_k-3)$ and $m_k = \frac{1}{2}(q_k-1)$, and we define a divisor
  $$D_k = d_kH - m_k E.$$
  Then $D_k$ is an $($integral\,$)$ effective divisor with $\chi(D_k) = 1$ and $2B\cdot D_k < B\cdot K$.

\item\label{thm-10p-2} Suppose $D$ is an effective divisor with $\chi(D) \geq 1$ and there is an ample divisor $A_t$ with $2A_t\cdot D < A_t\cdot K$.  Then there is a positive odd integer $k$ such that $D = D_k$.
\end{enumerate}
\end{theorem}

\begin{proof}
\eqref{thm-10p-1}~ The continued fraction expansion of $\sqrt{n}$ is as follows:
$$ \sqrt{10} = [3;\overline 6], \quad \sqrt{11} = [3; \overline{3,6}], \quad \sqrt{12} = [3;\overline{2,6}].$$
We write $a = 6/(n-9)$, so that in every case $\sqrt{n} = [3;\overline{a,6}]$. Then $\frac{p_1}{q_1} = \frac{3}{1}$, and for odd $k\geq 3$, we have the recurrence
\begin{align*}
p_k &= 6p_{k-1}+p_{k-2}, \\
q_k &= 6q_{k-1}+q_{k-2}.  
\end{align*}
Thus for odd $k$, we have $p_k\equiv p_{k-2} \pmod 2$ and $q_k \equiv q_{k-2} \pmod 2$, so $p_k$ and $q_k$ are both odd.  Therefore, $D_k$ is an integral divisor.

To show that $D_k$ is effective, we will show that $\chi(D_k) = 1$.  
Since $\chi(D_k) = 1+ \frac{1}{2} D_k\cdot (D_k-K)$, this is equivalent to proving the equality
$$p_k^2 - nq_k^2 = (2D_k-K)^2 = 4D_k\cdot (D_k-K) + K^2 = K^2 = 9-n.$$
For $k=1$, the equality is trivial.  From the fact that the continued fraction expansion of $\sqrt{n}$ has period dividing $2$, we can manipulate continued fractions to see that
$$
\frac{p_{k+2}}{q_{k+2}} = \frac{(3a+1)p_k+(9a+6)q_k}{ap_k + (3a+1)q_k}.
$$
Here the numerator and denominator are already coprime because the matrix $\begin{psmallmatrix} 3a+1 & 9a+6\\ a & 3a+1\end{psmallmatrix}$ has determinant $1$, so for $k$ odd, we have the recurrence
\begin{align*}
  p_{k+2} &= (3a+1)p_k +(9a+6)q_k,\\
  q_{k+2} &= ap_k + (3a+1)q_k.
\end{align*} 
Then we compute
\begin{align*}
p_{k+2}^2-nq_{k+2}^2 &= ((3a+1)p_k+(9a+6)q_k)^2-n(ap_k+(3a+1)q_k)^2 \\
&=(a(6-(n-9)a)+1)p_k^2 \\
&\hphantom{=(}+2(3a+1)(6-(n-9)a)p_kq_k\\
&\hphantom{=(}+((9a+6)(6-(n-9)a)-n)q_k^2
\\&= p_k^2-nq_k^2.
\end{align*}
Here the second equality is a direct computation, and the third follows since $a = 6/(n-9)$.   Thus $p_k^2-nq_k^2 = 9-n$ holds in every case.

Next we show that $2B\cdot D_k < B\cdot K$.  Let $t = nq_k/p_k$.  Since $p_k/q_k < \sqrt n$, we have $t > \sqrt{n}$.  Then $2D_k$ and $K$ have the same $A_t$-slope:
$$A_t \cdot (2D_k-K) = tp_k-nq_k=0.$$
As $t$ decreases to $\sqrt{n}$, we get the required inequality.

\eqref{thm-10p-2}~ Let $D$ be an effective divisor with $\chi(D) \geq 1$ and $2A_t\cdot D < A_t\cdot K$.  By Proposition~\ref{prop-equal}, it takes the form $D = dH - mE$.  Since $\chi(D) = 1$, we have $D\cdot (D-K)=0$, and
$$(2d+3)^2-n(2m+1)^2=(2D-K)^2 = 4D\cdot(D-K)+K^2=K^2=9-n.$$
Thus $(x,y) = (2d+3,2m+1)$ is a positive solution to the generalized Pell's equation
$$x^2-ny^2 = N$$
with $N = 9-n<0$.  It is known  that when $|N| < \sqrt n$ (which holds here since $10\leq n \leq 12$), every positive solution $(x,y)$ to this equation is of the form $(x,y)= (p_k,q_k)$ (see \cite[Theorem~20, p.~204]{Shockley}).
For $k$ even, we have $p_k^2-nq_k^2 >0$.  Therefore, there must be an odd integer $k$ such that $D = D_k$.
\end{proof}

When $10\leq n\leq 12$, we can now give a complete list of all the possible curves $D$.  They are closely related to the odd convergents in the continued fraction expansion of $\sqrt{n}$.

\begin{remark}
  For $n=10$, the convergents in the continued fraction expansion of $\sqrt{10}$ are
  $$\frac{3}{1},\,\frac{19}{6},\,\frac{117}{37},\,\frac{721}{228},\,\frac{4443}{1405},\,\frac{27379}{8658},\,\frac{168717}{53353},\,\ldots.$$
  The odd terms in this sequence give the corresponding divisors
  $$\OO,\quad 57H-18E,\quad 2220H-702E,\quad 84357H - 26676 E, \quad \ldots. $$

When $n=11$, the convergents in the continued fraction expansion of $\sqrt{11}$ are  
$$
\frac{3}{1},\,\frac{10}{3},\,\frac{63}{19},\,\frac{199}{60},\,\frac{1257}{379},\,\frac{3970}{1197},\,\frac{25077}{7561},\,\ldots,
$$
and the corresponding divisors are
$$\OO,\quad 30H-9E,\quad 627H - 189E,\quad 12537H - 3780E, \quad \ldots.$$

Finally,  when $n=12$, the convergents in the continued fraction expansion of $\sqrt{12}$ are
$$\frac{3}{1},\,\frac{7}{2},\,\frac{45}{13},\,\frac{97}{28},\,\frac{627}{181},\,\frac{1351}{390},\,\frac{8733}{2521},\,\ldots,$$
with corresponding divisors
$$\OO,\quad 21H-6E,\quad 312H-90E,\quad 4365H - 1260E, \quad \ldots.$$
\end{remark}

\subsection{Thirteen through seventeen points}\label{ssec-types}
For the blowup at $13$ to $17$ points, divisors $D$ must still be balanced, but they are no longer required to have equal multiplicities.  This makes the classification more complicated, but it is still tractable in each case.  It becomes necessary to solve several quadratic Diophantine equations, according to how unequal the multiplicities can be.  We discuss the cases $n=13$ and $n=16$ in more detail; the remaining cases can be handled identically.

\begin{example}
  Let $n=13$.  Here we classify the effective divisors $D$ such that $\chi(D)\geq 1$ and $2A_t\cdot D\leq A_t\cdot K$ for some ample divisor $A_t$.  We know that $D$ must have $\chi(D) = 1$ and that $D$ is balanced.  Furthermore, $D$ is at most one step away from having equal multiplicities (see Proposition~\ref{prop-equal}).  Then $D$ has the form $D = dH - mE'-(m-1)E''$, where $E'$ is a sum of $n-k$ exceptional divisors, $E'+E'' = E$, and $k$ is one of $0,1$, or $12$.  We must have
  $$(2D-K)^2 = 4D\cdot(D-K)+K^2 =K^2= -4.$$
  Considering the three possible values for $k$ separately  gives us three possible Diophantine equations to solve.  Conversely, note that $\chi(D)$ can be determined from $(2D-K)^2$, so that each solution to the Diophantine equations will necessarily give a divisor $D$ with $\chi(D)=1$.  Proposition~\ref{prop-ineq} furthermore shows that the solutions with $d\geq 0$ and $m\geq 1$ give divisors $D$ satisfying $2B\cdot D < B\cdot K$, and so $2A_t\cdot D\leq A_t\cdot K$ for some ample $A_t$ if we are assuming the Nagata conjecture.

  In more detail, we can expand the equation $(2D-K)^2 = 9-n$ to obtain
  $$(2d+3)^2-13(2m+1)^2+8km=-4.$$
  Specializing to the values $k=0,1,12$ gives the three associated equations:
  $$
  \begin{array}{rcrcl}
k=0: & &(2d+3)^2-13(2m+1)^2&=&-4,\\
k=1: & &(2d+3)^2-13(2m+1)^2+8m&=&-4,\\
k=12: & &(2d+3)^2-13(2m+1)^2+96m&=&-4.\\
  \end{array}
  $$
Quadratic Diophantine equations like this can be solved using Lagrange's method to transform them to generalized Pell's equations.  We used the online Alpertron generic two integer variable equation solver\footnote{https://www.alpertron.com.ar/QUAD.HTM} to find all the solutions in each case.

\emph{Case $k=0$}. In this case, there are four fundamental solutions $(d,m)=(-3,0)$, $(0,0)$, $(0,-1)$,  and $(-3,-1)$ (note that the third and fourth solutions can be obtained from the first two by Serre duality).  Given a solution $(d,m)$, new solutions can be obtained by applying the transformation
$$(d,m) \longmapsto (649d+2340m+2142,180d+649m+594)$$
or the inverse of this transformation.
In this way, each fundamental solution gives rise to infinitely many solutions indexed by the integers.  The solutions fit into the following chains:
$$ \cdots \longmapsto (-2782263,771660)\longmapsto (-2145,594)\longmapsto (-3,0) \longmapsto (195,54)\longmapsto (255057,70740)\longmapsto \cdots,$$
$$\cdots \longmapsto (-255060,70740)\longmapsto (-198,54)\longmapsto (0,0) \longmapsto (2142,594)\longmapsto (2782260,771660)\longmapsto \cdots,$$
$$\cdots \longmapsto\mkern-1mu (2782260,-771661)\longmapsto\mkern-1mu (2142,-595)\longmapsto\mkern-1mu (0,-1)\longmapsto\mkern-1mu (-198,-55)\longmapsto\mkern-1mu (-255060,-70741)\longmapsto \cdots, $$
$$\cdots \longmapsto\mkern-1mu (255057,-70741) \mkern-1mu\longmapsto\mkern-1.5mu (195,-55)\mkern-1mu\longmapsto\mkern-1.5mu (-3,-1)\mkern-1mu\longmapsto\mkern-1.5mu (-2145,-595)\longmapsto\mkern-1.5mu (-2782263,-771661)\mkern-1mu\longmapsto\mkern-1.5mu \cdots.$$
Out of all the solutions, the ones giving rise to effective divisors with $2B\cdot D < B\cdot K$ are the rightward chains 
$$(0,0)\longmapsto (2142,594)\longmapsto (2782260,771660)\longmapsto \cdots,$$ 
$$(195,54)\longmapsto (255057,70740)\longmapsto(331065735,91821114)\longmapsto \cdots. $$
The corresponding divisors $D$ come from two infinite families beginning with
\begin{equation}\tag{I}\label{type1}
  \OO,\quad 2142H-594E,\quad 2782260H-771660E,\quad \ldots
\end{equation}
and
\begin{equation}\tag{II}\label{type2}
  195H - 54E, \quad 255057H-70740E, \quad 331065735H-91821114E, \quad \ldots .
\end{equation}

\emph{Case $k=1$}. This time there are two fundamental solutions $(-3,0),(0,0)$, and the recurrence transformation is
$$(d,m)\longmapsto (-649d-2340m-1965,-180d-649m-545).$$
The fundamental solutions fit into the chains
\begin{align*}
  \cdots&\longmapsto (-2548623,706860)\longmapsto (1962,-545)\longmapsto (-3,0)\longmapsto (-18,-5)\longmapsto (21417,5940)\longmapsto \\
  &\longmapsto (-27801198,-7710665) \longmapsto (36085931637,10008436680)\longmapsto \cdots
\end{align*}
and
\begin{align*}
  \cdots & \longmapsto (-21420,5940)\longmapsto (15,-5)\longmapsto (0,0)\longmapsto (-1965,-545)\longmapsto (2548620,706860)\longmapsto\\
  &\longmapsto (-3308108745,-918504285)\longmapsto (4293922600440,1190919854520)\longmapsto \cdots.
\end{align*}
In each case, the geometrically relevant solutions are the ones obtained from the fundamental solutions by applying the recurrence an even number of times.  Note that $(0,0)$ is actually a geometrically relevant solution, corresponding to the divisor $E_{13}$.  Up to permuting the exceptional divisors, the corresponding divisors $D$ form two infinite families beginning with
\begin{equation}\tag{III}\label{type3}
  \begin{array}{c}
    21417H-5940E_{1,\ldots,12}-5939E_{13}, \\
    36085931637H-10008436680E_{1,\ldots,12}-10008436679E_{13}, \quad \ldots\end{array}
\end{equation}
and 
\begin{equation}\tag{IV}\label{type4}
  E_{13}, \quad 2548620H-706860E_{1,\ldots,12}-706859 E_{13}, \quad \ldots.
\end{equation}

\emph{Case $k=12$}. This case is very similar to the previous.  The fundamental solutions are again $(-3,0)$ and $(0,0)$, and there is a recurrence
$$(d,m)\longmapsto (-649d-2340m+15,-180d-649m+5).$$
We get chains of solutions
$$\cdots \longmapsto(0,0)\longmapsto (15,5)\longmapsto (-21420,-5940)\longmapsto(27801195,7710665)\longmapsto \cdots$$
and
$$\cdots\longmapsto (-3,0)\longmapsto (1962,545)\longmapsto (-2548623,-706860)\longmapsto (3308108742,917504285)\longmapsto \cdots.$$
In each case, applying the transformation an odd number of times to the fundamental solution gives a solution of geometric significance, and we get two infinite families of divisors beginning with the divisors
\begin{equation}\tag{V}\label{type5}
  15H-5E_1 - 4E_{2,\ldots,13}, \quad 27801195H - 7710665E_1-7710664E_{2,\ldots,13}, \quad \ldots
\end{equation}
and
\begin{equation}\tag{VI}\label{type6}
  1962H-545E_1-544E_{2,\ldots,13}, \quad 3308108742H-917504285E_1-917504284 E_{2,\ldots,13},\quad \ldots.
\end{equation}
The solution $(0,0)$ corresponds to the divisor $E_{2,\ldots,13}$, which does not satisfy $2B\cdot D< B\cdot K$.

\begin{theorem}\label{thm-D13}
Let $n=13$.  If\, $D$ is an effective divisor with $\chi(D)\geq 1$ such that $2A_t\cdot D <A_t\cdot K$ holds for some ample divisor $A_t$, then $D$ comes from one of the six infinite families \eqref{type1}--\eqref{type6} discussed above.  Conversely, if the Nagata conjecture holds for $n=13$, then the divisors in these six families are precisely the effective divisors $D$ with $\chi(D)=1$ such that $2B\cdot D < B\cdot K$.
\end{theorem}
\end{example}

\begin{example}
Let $n=16$.  Since the Nagata conjecture is true for $16$ points, our results here are sharper, and we can completely classify effective divisors $D$ such that $\chi(D) \geq 1$ and $2B\cdot D < B\cdot K$.

The method is the same as the method for $n=13$: we find several quadratic Diophantine equations and determine all their solutions.  However, these Diophantine equations turn out to only have finitely many solutions, so our answer is considerably more concrete.

Suppose $D$ has $\chi(D)\geq 1$ and $2B\cdot D< B\cdot K$.  We know $\chi(D) = 1$ and $D$ is balanced.  Furthermore, by Proposition~\ref{prop-equal}, it is at most one step away from having equal multiplicities.  Write $D = dH - mE'-(m-1)E''$ as in the previous example.  Then the Diophantine equation $(2D-K)^2 = 9-n$ becomes
$$(2d+3)^2-16(2m+1)^2+8km=-7.$$ 
We need to find all solutions for $k=0,1,15$.

Note that the pairs $(d,m) = (0,0)$ and $(-3,0)$ are solutions for every $k$.  The solution $(-3,0)$ is not geometrically relevant.  On the other hand, $(0,0)$ corresponds to a sum of $k$ exceptional divisors.  This will satisfy $2B\cdot D < B\cdot K = 4$ only if $k$ is $0$ or $1$.  In what follows, we ignore these solutions and search for additional solutions.

\emph{Case $k=0$}.
The only additional solutions are $(0,-1),(-3,-1)$.  Neither is geometrically relevant.

\emph{Case $k=1,15$}. There are no additional solutions.

We summarize the discussion in the following theorem.

\begin{theorem}\label{thm-D16}
Let $n=16$.  The only effective divisors $D$ with $\chi(D) \geq 1$ and $2B\cdot D < B\cdot K$ are $\OO$ and the exceptional divisors $E_i$.
\end{theorem}
\end{example}

\subsection{Twenty-five points}
Here we classify the effective divisors $D$ with $\chi(D) \geq 1$ and $2B\cdot D < B\cdot K$ when $n=25$.  Again in this case, the Nagata conjecture holds and $B = 5H - E$ is nef.

\begin{example}
Let $n=25$, and suppose $D$ is an effective divisor with $\chi(D) \geq 1$ and $2B\cdot D<B\cdot K$.  By Proposition~\ref{prop-chiUpper}, we have $1\leq \chi(D) \leq 2$.  

First suppose $\chi(D) = 2$.   By Proposition~\ref{prop-balanced}, we know that $D$ has balanced multiplicities, and by Proposition~\ref{prop-equal}, we know that $D$ is at most two steps away from having equal multiplicities.  The equality $\chi(D)=2$ is equivalent to $(2D-K)^2 = 8\chi(D)-8+K^2=-8$.  Writing $D = dH - mE'-(m-1)E''$ as in the previous example, we have the Diophantine equation
$$(2d+3)^2 - 25(2m+1)^2 + 8km = -8,$$
and we must find solutions when $k = 0,1,2,23,24$. The only solutions are when $(d,m,k)$ takes the following values:
$$(1,-1,1),\,(-4,-1,1),\,(-4,1,24),\,(1,1,24).$$
Only the solution $(1,1,24)$ is relevant, and it corresponds to the divisors of the form $H-E_i$ for some $i$.

\begin{theorem}\label{thm-D25}
Let $n=25$.  The only effective divisors $D$ with $\chi(D)\geq 2$ and $2B\cdot D < B\cdot K$ are the divisors $H-E_i$.
\end{theorem}

In the rest of the paper, we will only need our classification when $\chi(D) =2$, but we briefly state the classification for $\chi(D) = 1$ for completeness.  Here we have to allow for the possibility that $D$ does not have balanced multiplicities since Proposition~\ref{prop-balanced} only guarantees that the multiplicities of $D$ are at most one step away from being balanced.  If $D$ does not have balanced multiplicities, then  the balanced divisor $D_1$ obtained from $D$ will satisfy $\chi(D_1)\geq 2$ and $2B\cdot D_1 < B\cdot K$.  But then we must actually have $\chi(D_1) = 2$, and $D_1$ must be of the form $H-E_i$ from the previous case.  For the divisor $D$ to have $\chi(D) = 1$ and become equal to $H-E_i$ after a single rebalancing step, we must have that $D$ is of the form $H-E_i-E_j+E_k$ for distinct $i,j,k$.

The other possibility is that $D$ does have balanced multiplicities.  By Proposition~\ref{prop-equal}, the multiplicities of~$D$ are at most four steps away from being equal.  Using the methods from the previous classifications, we list all the divisors $D$ satisfying $\chi(D) \geq 1$ and $2B\cdot D < B\cdot K$ in the following table, ordered so that $t_D$ is decreasing.  For brevity, we only list divisor classes up to permutations of the exceptional divisors. 
$$
\begin{array}{c|ccc}
D & 2B\cdot D & \chi(D) & t_{D}\\\hline
\OO & 0 & 1 & 25/3\\
E_1 & 2 & 1 & 23/3\\
E_{12} & 4 & 1 & 7\\
E_{123} & 6 & 1 & 19/3\\
H-E_{12} & 6 & 1 & 29/5\\
E_{1234} & 8 & 1 & 17/3\\
H-E_{12}+E_3 & 8 & 1 & 27/5\\
H-E_1 & 8 & 2 & 27/5\\
6H-2E_1-E_{2,\ldots,25} & 8 & 1 & 77/15\\
\end{array}
$$
\end{example}

\section{Cohomological properties of \texorpdfstring{$\boldsymbol{D}$}{$D$} and the SHGH conjecture}\label{sec-cohomology}
When we analyze stability, we will need to know various cohomological properties about the divisors $D$.  To analyze these, we will assume the SHGH conjecture.

\begin{lemma}\label{lem-irr}
$($Assume SHGH\,$)$ Suppose $10\leq n\leq 16$.
  Then there is no reduced, irreducible curve $D$ satisfying
  $$4B\cdot D < B\cdot K.$$
\end{lemma}

\begin{proof}
Suppose there is such a reduced, irreducible curve $D$.  By the SHGH conjecture, $\chi(D) \geq 1$.  Since $4B \cdot D < B\cdot K$, we also have $2B \cdot D < B \cdot K$, so Proposition~\ref{prop-chiUpper} gives $\chi(D) = 1$.  We now consider three cases based on the value of $D^2$.

\emph{Case {\rm 1:} $D^2>0$}.  In this case, $2D$ is an effective divisor satisfying $2B\cdot 2D < B\cdot K$, and
$$\chi(2D) = 2\chi(D) +D^2 - 1\geq 2.$$
This contradicts Proposition~\ref{prop-chiUpper}.

\emph{Case {\rm 2:} $D^2 < 0$.}  In this case, the SHGH conjecture implies $D$ is a $(-1)$-curve, so $D^2 = -1$ and $D\cdot K = -1$.  We write
$$B = -K + (\sqrt{n}-3)H$$
and conclude that
$$B \cdot D = 1+ (\sqrt{n}-3)H\cdot D \geq 1.$$
Since $n\leq 16$, we have $B\cdot K = n-3\sqrt{n} \leq 4$, so $4B\cdot D \geq B\cdot K$. 

\emph{Case {\rm 3:} $D^2 = 0$.} Since $\chi(D) = 1$, we find that $D\cdot K=0$.  Thus $D$ has arithmetic genus $1$, and $D\cdot H \geq 3$.      Using the above decomposition of $B$ shows
$$B\cdot D = (\sqrt{n}-3)H \geq 3(\sqrt{n}-3).$$
Then $12(\sqrt{n}-3) = 4B \cdot D \geq B\cdot K = n-3\sqrt{n}$ is seen to hold as long as $9\leq n\leq 144$.
\end{proof}

\begin{theorem}\label{thm-irr}
$($Assume SHGH\,$)$ Suppose $10\leq n\leq 16$.  If\, $D$ is an effective divisor with $2B\cdot D < B\cdot K$, then $D$ is a reduced, irreducible curve with $h^0(\OO(D)) = 1$ and $h^1(\OO(D))=h^2(\OO(D))=0$.
\end{theorem}

\begin{proof}
If $D$ is not reduced and irreducible, then we can write it as $D = D' + D'' + D'''$, where $D'$ and $D''$ are reduced and irreducible curves and $D'''$ is effective (possibly empty).  Since $2B\cdot D<B\cdot K$, at least one of $4B\cdot D'$ or $4B \cdot D''$ is less than $B\cdot K$, contradicting Lemma~\ref{lem-irr}.  Therefore, $D$ is a reduced and irreducible curve.  By the SHGH conjecture, $\chi(D) \geq 1$, so Proposition~\ref{prop-chiUpper} gives $\chi(D) = 1$.  The cohomology of $D$ then follows from the SHGH conjecture.
\end{proof}

When analyzing the tangent space to the moduli space, we will need to understand the cohomology of $\OO(2D)$, so we now compute this.

\begin{corollary}\label{cor-doubleCohomology}
$($Assume SHGH\,$)$ 
Let $10\leq n\leq 16$, and suppose $D$ is an effective divisor with $2B\cdot D < B\cdot K$.  
\begin{enumerate}
\item\label{c-dC-1} If\, $D$ is not an exceptional divisor $E_i$, then $h^0(2D)\geq 1$ and $h^1(2D) =  h^2(2D) = 0$. 
\item\label{c-dC-2} If\, $D$ is an exceptional divisor $E_i$, then $h^0(2E_{i}) = h^1(2E_{i}) = 1$ and $h^2(2E_i) = 0$.
\end{enumerate}
\end{corollary}

\begin{proof}
Statement~\eqref{c-dC-2} is clear.

\eqref{c-dC-1}~ By Theorem~\ref{thm-irr}, we just need to make sure that $D$ is not a $(-1)$-curve.  If $D$ is a $(-1)$-curve, so $D^2= -1$ and $D\cdot K = -1$, then as in the proof of Lemma~\ref{lem-irr}, we have
$$2B\cdot D = 2+(2\sqrt{n}-6)H\cdot D < n-3\sqrt{n}=B\cdot K.$$
This implies
$$H\cdot D < \frac{n-3\sqrt{n}-2}{2\sqrt{n}-6},$$
but since $10 \leq n \leq 16$, this inequality implies $H\cdot D < 1$.  The only such $(-1)$-curves are the exceptional divisors.
\end{proof}

Finally, we also compute the cohomology of $\OO(2D-K)$ because 
bundles of type $D$ are parameterized by an extension class in $\Ext^1(K(-D),\OO(D)) \cong H^1(\OO(2D-K))$.

\begin{proposition}\label{prop-2DminusK}
$($Assume SHGH\,$)$ Let $10\leq n\leq 16$, and suppose $D$ is an effective divisor with $2B\cdot D < B\cdot K$.  The line bundle $\OO(2D-K)$ has no $h^0$ or $h^2$, and its $h^1$ is nonzero unless $D$ is trivial and $n=10$.
\end{proposition}

\begin{proof}
  Since $2B\cdot D < B \cdot K$, we have $A_t\cdot (2D - K)<0$ for $t$ slightly greater than $\sqrt{n}$, so $2D-K$ is not effective.  There is no $h^2$ because the coefficient of $H$ is positive.  Thus $\chi(\OO(2D-K)) \leq 0$, and it remains to show that the inequality is strict unless $D$ is trivial and $n=10$.  We use Riemann--Roch to write
  $$\chi(\OO(2D-K)) = 1+ \frac{1}{2}(2D-K)\cdot(2D-2K) = 1+\frac{1}{2}(2D-K)^2+\frac{1}{2}(2D-K)\cdot(-K).$$
  Since $\chi(D) = 1$, we know that
  $$(2D-K)^2 = K^2 = 9-n \leq -1.$$
  Since $-K = A_3$, the inequality $2B\cdot D < B\cdot K$ also implies $2A_3\cdot D< A_3\cdot K$, so
  $$(2D-K)\cdot (-K)\leq -1.$$
  In order to have $\chi(\OO(2D-K)) = 0$, we must then have both $n=10$ and $(2D-K)\cdot(-K) = -1$.

  When $n=10$, we have fully classified the possible divisors $D$ using the odd convergents in the continued fraction expansion of $\sqrt{10}$.  Suppose $k\geq 3$ is odd and $D = D_k$ in the classification of Theorem~\ref{thm-10points}.  Then
  $$2D_k - K = p_kH - q_kE$$
  and
  $$(2D_k-K)\cdot(-K) = 3p_k-10q_k = q_k\left(3\frac{p_k}{q_k}-10\right)<q_k\left(3\sqrt{10} - 10\right) \approx -0.51 q_k.$$
  Since $k\geq 3$, we have $q_k \geq 37$, so $(2D-K)\cdot (-K)$ is  less than $-1$.
\end{proof}

\begin{remark}\label{rem-growth}
  Let $10\leq n\leq 12$, and let $D_k$ be the divisor of Theorem~\ref{thm-10points}.  To avoid the special case $D= \OO$, we may as well assume $k\geq 3$.  By a similar analysis, we can give a formula for $\chi(\OO(2D_k-K))$ which shows more explicitly how this quantity grows with $k$.  As in the proof of the lemma, we have
  $$\chi(\OO(2D_{k}-K)) = 1 + \frac{9-n}{2}+\frac{1}{2}(3p_k-nq_k).$$
  We estimate the final term from above as
  $$3p_k-nq_k =q_k \left( 3\frac{p_k}{q_k}-n\right) < q_k\left(3\sqrt{n} - n\right).$$
  The error in this approximation is only
  $$q_k(3\sqrt{n}{-n}) - q_k\left(3\frac{p_k}{q_k}-n\right) = 3q_k\left(\sqrt{n}-\frac{p_k}{q_k}\right)<1.$$
  Here in the final step we have used the well-known fact, see \cite[Section IV.6]{Davenport}, that the convergents of the continued fraction expansion of a real number $x$ satisfy
  $$\left|x - \frac{p_k}{q_k}\right| < \frac{1}{q_kq_{k+1}},$$
  together with the observation that $q_{k+1}$ is considerably greater than $3$ since $k\geq 3$.  Thus in fact
  $$3p_k - nq_k = \left\lfloor q_k(3\sqrt{n}-n)\right\rfloor$$
  and
  $$\chi(\OO(2D_k-K)) = \frac{1}{2}\left(11 - n + \left\lfloor q_k\left(3 \sqrt{n} - n\right)\right\rfloor\right).$$
  The Euler characteristic $\chi(\OO(2D_k-K))$ has the same growth rate as the denominators of the continued fractions of $\sqrt{n}$.
\end{remark}

\section{Stability, components, and the SHGH conjecture}\label{sec-components}

Throughout this section, we assume the SHGH conjecture holds, and we prove our main theorems on the components of the moduli spaces $M_{A_t}(\bv)$, where $\bv = (r,c_1,\chi) = (2,K,2)$ and $10\leq n\leq 16$.  These spaces are particularly nice because this is the maximal possible value of the Euler characteristic $\chi$.  We will see that this causes the moduli spaces to be smooth and the irreducible components to be disjoint from one another.  It also causes every semistable sheaf to be a vector bundle, so the moduli space is stratified by the type of a bundle.

\begin{lemma}\label{lem-free}
Every $\mu_{A_t}$-semistable sheaf $V$ with rank $2$ and $c_1 =K$ has $\chi(V)\leq 2$.  In particular, every $\mu_{A_t}$-semistable sheaf of character $\bv = (2,K,2)$ is a vector bundle.
\end{lemma}

\begin{proof}
  Let $V$ be a $\mu_{A_t}$-semistable sheaf with rank $2$ and $c_1 = K$, and suppose $\chi(V) \geq 2$.  Let $W = V^{**}$ be the double dual.  Then there is an exact sequence
  $$0\lra V\lra W\lra T\lra 0,$$
  where $T$ is torsion and (at most) zero-dimensional.  Then $W$ is a $\mu_{A_t}$-semistable vector bundle with rank $2$ and $c_1 = K$.  Furthermore, $\chi(W) \geq \chi(V)\geq 2$, so $W$ has a type $D$.  By Propositions~\ref{prop-Dstablenecessary} and~\ref{prop-chiUpper}, we have $\chi(\OO(D))\leq 1$ and therefore $\chi(W) \leq 2$.  But then $\chi(W) = \chi(V) = 2$, so $\chi(T) = 0$ and $T=0$ and $V= W$ is a vector bundle.   
\end{proof}

For the rest of the section, let $10\leq n\leq 16$ and $\bv = (r,c_1,\chi) = (2,K,2)$.  We let $D$ be a (possibly trivial) effective divisor satisfying $\chi(D) = 1$ and $2B\cdot D < B\cdot K$.  
We denote by $V$ any bundle of character $\bv$ and type $D$ given by a \emph{nonsplit}
extension
$$0\lra \OO(D) \lra V\lra K(-D) \lra 0.$$ 
  
\subsection{Stability}
We next study the stability of $V$.  Stability behaves slightly differently for the trivial type $D = \OO$ and nontrivial types, where $D$ is actually effective.  We focus on the latter case first.

\begin{proposition}\label{prop-Dstable}
$($Assume SHGH\,$)$  If\, $D$ is nontrivial, then $V$ is $\mu_{A_{t}}$-stable for all $t$ such that $\sqrt{n} < t < t_{D}$.
\end{proposition}

\begin{proof}
  Suppose that $\sqrt n < t < t_{D}$ and that $V$ is not $\mu_{A_t}$-stable.  Then there is a saturated line subbundle $L \subset V$ with $\mu_{A_t}(L) > \mu_{A_t}(V)$.  The bundle $V$ is a nonsplit extension
  $$0\lra \OO(D) \lra V\lra K(-D)\lra 0,$$
  and the assumption $t < t_D$ gives $\mu_{A_t} (\OO(D)) < \mu_{A_t}(K(-D))$.  Then the composition $L\to V\to K(-D)$ must be nonzero, for otherwise there would be an inclusion $L\to \OO(D)$.  Thus $L$ takes the form $K(-D')$ for an effective divisor $D'$, and $D'-D$ must be nontrivial (otherwise $V$ is split) and effective.  In particular, by summing a curve of class $D'-D$ and one of class $D$, it follows that $D'$ can be represented by a nonintegral curve.  However, the inequality $\mu_{A_t}(L) > \mu_{A_t}(V)$ reads
  $$A_t\cdot (K-D') > \frac{1}{2}A_t\cdot K,$$
  or
  $$2A_{t}\cdot D'  < A_{t}\cdot K.$$
  This implies $2B\cdot D' < B\cdot K$.  Additionally, $V$ fits in an exact sequence
  $$0\lra K(-D')\lra V\lra \OO(D')\lra 0,$$
  which forces $\chi(D') \geq 1$.  This all contradicts Theorem~\ref{thm-irr}.
\end{proof}

\begin{proposition}\label{prop-Ostable}
$($Assume SHGH\,$)$ Suppose $D = \OO$ is trivial and $V$ is a nonsplit bundle of type $\OO$.  $($This implies $n\geq 11$.$)$
\begin{enumerate}
\item\label{pO-1} If\, $n=11$ or $12$, then $V$ is $\mu_{A_t}$-stable for every $t$ with $\sqrt{n} <t < t_\OO = \frac{n}{3}$.
\item\label{pO-2} If\, $13\leq n \leq 16$, then the same result is true except that there are $n$ points in $\P\Ext^1(K,\OO)$ parameterizing bundles $V$ which are only $\mu_{A_t}$-stable if $t$ satisfies $\frac{n-2}{3} = t_{E_{1}}< t < t_{\OO}$.  These points are given by the images of the inclusions of $1$-dimensional spaces
  $$\Hom(K,\OO_{E_i}(-1))\lra \Ext^1(K,\OO).$$
\end{enumerate}
\end{proposition}

\begin{proof}
  We begin as in the previous proof.  If there is a $t$ with $\sqrt{n} < t < t_\OO$ such that $V$ is not $\mu_{A_t}$-stable, then there is a saturated line subbundle of $V$ of the form $K(-D')$, where $D'$ is a nontrivial effective divisor satisfying $2B\cdot D' < B\cdot K$ and $\chi(D') = 1$.  Also, $V$ fits as an extension
  $$0\lra K(-D') \lra V \lra \OO(D')\lra 0,$$
  and this extension cannot be split since then $V$ would have both type $\OO$ and type $D'$.  However, $\Ext^1(\OO(D'),K(-D')) = H^1(K(-2D')) = H^1(\OO(2D'))$.  By Corollary~\ref{cor-doubleCohomology}, the only way this is nonzero is if $D'$ is one of the exceptional divisors $E_i$.  

\eqref{pO-1}~ When $n=11$ or $12$, the exceptional divisors do not satisfy $2B\cdot E_i < B\cdot K$, so $V$ is always stable.

\eqref{pO-2}~ Suppose $13\leq n \leq 16$.  Fix one of the exceptional divisors; without loss of generality say it is $E_1$.  We seek to describe the extension classes $e\in \Ext^1(K,\OO)$ such that the corresponding bundle $V$ admits a nonzero map to $\OO(E_1)$.  
 
From the defining sequence
$$0\lra \OO \lra V\lra K\lra 0$$
and the restriction sequence
$$0\lra \OO\lra \OO(E_1)\lra \OO_{E_1}(-1)\lra 0,$$
we get the following  commuting diagram which has exact rows and columns:  
$$ 
\xymatrix{
&&&\Ext^1(K,\OO_{E_1}(-1))\\
\Hom(K,\OO(E_1)) \ar[r] & \Hom(V,\OO(E_1)) \ar[r]& \Hom(\OO,\OO(E_1)) \ar[r] & \Ext^1(K,\OO(E_1)) \ar[u]\\&\Hom(V,\OO) \ar[r]\ar[u] & \Hom(\OO,\OO) \ar[r]\ar[u] & \Ext^1(K,\OO)\ar[u]\\
&&0\ar[r]\ar[u]&\Hom(K,\OO_{E_1}(-1))\ar[u]\\
&&&\Hom(K,\OO(E_1))\rlap{.} \ar[u]
}
$$
We can compute many of the terms in this diagram to get a diagram
$$ 
\xymatrix{
&&& 0\\
0 \ar[r] & \Hom(V,\OO(E_1)) \ar[r]& \CC \ar[r] & \CC^{n-11} \ar[u]\\&0 \ar[r]\ar[u] & \CC \ar[r]\ar[u]^{\cong} & \CC^{n-10}\ar[u]\\
&&0\ar[r]\ar[u]&\CC \ar[u]\\
&&&0\rlap{.} \ar[u]
}
$$
Then $\Hom(V,\OO(E_1))$ is nonzero if and only if the map
$$\Hom(\OO,\OO(E_1))\lra \Ext^1(K,\OO(E_1))$$
is zero, which holds if and only if the composition
$$\Hom(\OO,\OO)\lra \Hom(\OO,\OO(E_1))\lra \Ext^1(K,\OO(E_1))$$
is zero.  This composition is the same as the composition
$$\Hom(\OO,\OO)\lra \Ext^1(K,\OO)\lra \Ext^1(K,\OO(E_1)),$$
and this is zero if and only if the image of the first map is contained in the kernel of the second.  The image of the first map is the $1$-dimensional subspace of $\Ext^1(K,\OO)$ determined by the bundle $V$, and the kernel of $\Ext^1(K,\OO)\to \Ext^1(K,\OO(E_1))$ is $1$-dimensional.  Therefore, $\Hom(V,\OO(E_1))$ is nonzero if and only if $V$ is the bundle defined by an extension class in the $1$-dimensional image of the canonical map
$$\CC \cong \Hom(K,\OO_{E_1}(-1))\lra \Ext^1(K,\OO).$$
Such a bundle is not $\mu_{A_t}$-stable for any $t$ with $t\leq t_{E_1}$ because the map $V\to \OO(E_1)$ would destabilize $V$.

Additionally, we note that if $V$ fits in the exact sequence
$$0\lra K(-E_i) \lra V\lra \OO(E_i)\lra 0,$$
then $\Hom(V,\OO(E_j)) = 0$ for $i\neq j$ since $n\geq 13$.  Therefore, the $n$ points in $\P\Ext^1(K,\OO)$ corresponding to such bundles are distinct.
\end{proof}

\subsection{Tangent space}
Now that we have determined when nonsplit bundles $V$ of type $D$ are stable, we investigate the components of the moduli space given by bundles of the various types.  The tangent space to the moduli space at $V$ is given by $\Ext^1(V,V)$, so we compute this space now.

\begin{lemma}\label{lem-ext}
$($Assume SHGH\,$)$ The spaces $\Ext^i(V,V)$ have dimensions given as follows: 
\begin{enumerate}
\item\label{le-1} If\, $D$ is not one of the exceptional divisors $E_i$, then we have
\begin{align*}
\hom(V,V) &= 1,\\
\ext^1(V,V) &= -\chi(2D-K)-1,\\
\ext^2(V,V) & = \chi(2D).
\end{align*}
\item\label{le-2} In particular, if $n\geq 11$ and $D = \OO$,  we get
\begin{align*}
\hom(V,V) &= 1,\\
\ext^1(V,V) &= n-11,\\
\ext^2(V,V) & = 1.
\end{align*}
\item\label{le-3} If\, $D = E_i$ (which forces $n\geq 13$), then we also have 
\begin{align*}
\hom(V,V) &= 1,\\
\ext^1(V,V) &= n-11,\\
\ext^2(V,V) &= 1.
\end{align*}
\end{enumerate}
\end{lemma}

\begin{proof}
  Let $V$ be a nonsplit bundle of type $D$, given by a nonsplit extension as
  $$0\lra \OO(D)\lra V\lra K(-D)\lra 0.$$
  We first apply $\Hom(-,\OO(D))$ to this sequence.  We display the dimensions $\ext^i(A,B)$ for the relevant pairs $(A,B)$ of objects in the following table:  
$$\begin{array}{c|cccccc}
&& (K(-D),\OO(D)) && (V,\OO(D)) && (\OO(D),\OO(D)) \\\hline
\hom &&0& \lra &0& \to &1\\
\ext^1 & \lra & -\chi(2D-K) & \lra &-\chi(2D-K)-1&\lra & 0\\
\ext^2 & \lra &0& \lra & 0& \lra & 0\\ 
\end{array}
$$
Here the first column of values $\ext^i(K(-D),\OO(D))$ are given by $h^i(\OO(2D-K))$, which were computed in Proposition~\ref{prop-2DminusK}.  The third column is clear, and the map $\Hom(\OO(D),\OO(D))\to \Ext^1(K(-D),\OO(D))$ is injective because $V$ is nonsplit.  The values of $\ext^i(V,\OO(D))$ follow.

Let $e$ be $0$ if $D$ is not exceptional, and let $e$ be $1$ if $D$ is exceptional.  Then by Corollary~\ref{cor-doubleCohomology}, the line bundle $\OO(2D)$ has $h^1(\OO(2D)) = e$ and $h^2(\OO(2D)) = 0$.   When we apply $\Hom(-,K(-D))$ to the sequence, we get the following cohomology since in the third column, we have $\ext^i(\OO(D),K(-D))) = h^i(K(-2D)) = h^{2-i}(\OO(2D))$: 
$$
\begin{array}{c|cccccc}
&& (K(-D),K(-D)) && (V,K(-D)) && (\OO(D),K(-D)) \\\hline
\hom &&1& \lra &1& \lra & 0\\
\ext^1 & \lra & 0 & \lra &e &\lra & e\\
\ext^2 & \lra &0& \lra & h^0(\OO(2D)) & \lra & h^0(\OO(2D))\\ 
\end{array}
$$
Finally, we apply $\Hom(V,-)$ to the sequence and get the following table: 
$$
\begin{array}{c|cccccc}
&& (V,\OO(D)) && (V,V) && (V,K(-D)) \\\hline
\hom &&0& \lra &1& \lra & 1\\
\ext^1 & \lra & -\chi(2D-K)-1 & \lra & -\chi(2D-K)-1+e &\lra & e\\
\ext^2 & \lra &0& \lra & h^0(\OO(2D)) & \lra & h^0(\OO(2D))\\ 
\end{array}
$$
Clearly $\hom(V,V) \geq 1$, but $\Hom(V,V)\to \Hom(V,K(-D))$ is an injection and hence an isomorphism.  The rest of the table is immediate.  This proves part~\eqref{le-1}, and specializing to the situations of~\eqref{le-2} and~\eqref{le-3}  completes the proof.
\end{proof}

\subsection{Components of nontrivial type}
In fact, each nontrivial type $D$ other than the exceptional divisors $E_i$ gives rise to a disjoint component of the moduli space $M_{A_t}(\bv)$ whenever $t < t_D$, and this component is a projective space of extensions.

\begin{theorem}\label{thm-componentD}
$($Assume SHGH\,$)$ Assume that $D$ is not trivial or one of the exceptional divisors $E_i$.  If $\sqrt{n} < t< t_D$, then the nonsplit bundles of type $D$ sweep out an irreducible component of\, $M_{A_t}(2,K,2)$ which is isomorphic to the projective space
  $$\P \Ext^1(K(-D),\OO(D)) \cong \P H^1(\OO(2D-K)) \cong \P^{-\chi(2D-K)-1}.$$
  This component is disjoint from all other components of the moduli space.
\end{theorem}

\begin{proof}
  By Proposition~\ref{prop-Dstable}, every nonsplit extension
  $$0\lra \OO(D) \lra V\lra K(-D)\lra 0$$
  of $K(-D)$ by $\OO(D)$ gives rise to an $A_t$-stable bundle of type $D$.  Then the universal extension over $\P \Ext^1(K(-D),\OO(D))$ induces a natural morphism
  $$\phi\colon\P \Ext^1(K(-D),\OO(D))\lra M_{A_t}(2,K,2).$$
  To complete the proof, it suffices to show that $\phi$ is an injection and an isomorphism on tangent spaces.

In the proof of Lemma~\ref{lem-ext}, we showed that $\Hom(V,K(-D))=\CC$. It follows that if $V$ fits as an extension of $K(-D)$ by $\OO(D)$ in two different ways, then the corresponding classes in $\Ext^1(K(-D),\OO(D))$ are multiples of each other.  Therefore, $\phi$ is injective.

Let $V$ be given by the extension class $e \in \Ext^1(K(-D),\OO(D))$, and let $[e]$ denote its image in $\P\Ext^1(K(-D),\OO(D))$.  We let
$$\tilde \phi\colon\Ext^1(K(-D),\OO(D))\sm\{0\}\lra M_{A_t}(2,K,2)$$
be the composition of the quotient map $\Ext^1(K(-D),\OO(D))\sm \{0\}\to \P \Ext^1(K(-D),\OO(D))$ and $\phi$.  Then the derivative $d \tilde \phi_e$ factors as the composition of the natural maps
$$\Ext^1(K(-D),\OO(D))\flra\alpha \Ext^1(K(-D),V)\flra\beta \Ext^1(V,V)$$
coming from applying various $\Hom$ functors to the defining sequence of $V$.  The map $\alpha$ fits into the sequence
$$\Ext^1(K(-D),\OO(D))\flra\alpha \Ext^1(K(-D),V) \lra \Ext^1(K(-D),K(-D)) = 0,$$
so $\alpha$ is surjective.  The map $\beta$ fits into the sequence
$$\Ext^1(K(-D),V)\flra\beta \Ext^1(V,V)\lra \Ext^1(\OO(D),V),$$
and therefore $\beta$ is surjective since we compute $\Ext^1(\OO(D),V) = H^1(V(-D))= 0$ from the sequence
$$H^1(\OO)\lra H^1(V(-D))\lra H^1(K(-2D)),$$
where $H^1(K(-2D)) = H^1(\OO(2D)) = 0$ by Corollary~\ref{cor-doubleCohomology}.

We conclude that $d\tilde \phi_e$ is surjective, with $\CC e$ contained in its kernel.  The tangent space to $\P \Ext^1(K(-D),\OO(D))$ at $[e]$ is naturally identified with $\Ext^1(K(-D),\OO(D))/\CC e$, and $d\tilde \phi_e$ factors through $d \phi_{[e]}$ to show that
$$d\phi_{[e]} \colon \Ext^{1}(K(-D),\OO(D))/\CC e \lra \Ext^1(V,V)$$
is surjective.  These spaces have the same dimension by Lemma~\ref{lem-ext}, so $d\phi_{[e]}$ is an isomorphism.  This completes the proof.\end{proof}

\subsection{The component of trivial type}
When $n=11$ or $12$, bundles of trivial type again sweep out a component.

\begin{theorem}\label{thm66}
$($Assume SHGH\,$)$ Suppose $n=11$ or $12$. For any $t$ with $\sqrt{n}<t<t_\OO$, the nonsplit bundles of type $\OO$ sweep out a component of $M_{A_t}(2,K,2)$ isomorphic to $\P\Ext^1(K,\OO) \cong \P^{n-11}$.
\end{theorem}

\begin{proof}
The proof is the same as that of Theorem~\ref{thm-componentD}, using Proposition~\ref{prop-Ostable}\eqref{pO-1} to establish stability of nonsplit extensions.
\end{proof}

On the other hand, bundles of type $E_i$ complicate the picture for $13\leq n\leq 16$.  Nevertheless, we completely identify the moduli space in this case.

\begin{theorem}\label{thm-componentO}
$($Assume SHGH\,$)$ Suppose $13\leq n \leq 16$.
\begin{enumerate}
\item\label{t-cO-1} For any $t$ with $t_{E_1}< t < t_\OO$, the nonsplit bundles of type $\OO$ sweep out a component of $M_{A_t}(2,K,2)$ isomorphic to $\P\Ext^1(K,\OO) \cong \P^{n-11}$.
\item\label{t-cO-2} For any $t$ with $\sqrt{n} < t < t_{E_1}$, the component of $M_{A_t}(2,K,2)$ containing stable bundles of type $\OO$ consists of all the stable bundles of type $\OO$ together with the nonsplit bundles of each type $E_i$.  This component is isomorphic to the blowup of the projective space $\P\Ext^1(K,\OO)$ at the $n$ points determined by the canonical inclusions of\, $1$-dimensional spaces
  $$\Hom(K,\OO_{E_i}(-1))\lra \Ext^1(K,\OO),$$
  and it is disjoint from all other components.
\end{enumerate}
\end{theorem}

\begin{proof}
\eqref{t-cO-1}~ The proof is the same as that of Theorem~\ref{thm66}.

\eqref{t-cO-2}~ We know from Theorem~\ref{thm-componentD} that components of the moduli space parameterizing bundles of types $D$ other than $\OO$ or $E_i$ are disjoint from any components which parameterize a bundle of type $\OO$ or $E_i$.  Thus we can let $M\subset M_{A_t}(2,K,2)$ be the subscheme parameterizing any bundles of type $\OO$ or $E_i$.  By Lemma~\ref{lem-ext}, the tangent space of $M$ has dimension $n-11$ at every point $V\in M$.  

Let us write $q_1,\ldots,q_n$ for the $n$ points in $\P\Ext^1(K,\OO)$ corresponding to the inclusions $\Hom(K,\OO_{E_i}(-1))\mkern-1mu\to\mkern-1mu \Ext^1(K,\OO)$.  Then $M$ is projective and contains the subvariety
$$U = \P\Ext^1(K,\OO)\sm\{q_1,\ldots,q_n\} \subset M,$$
which parameterizes the $\mu_{A_t}$-stable bundles of type $\OO$.  Let $Y_i\subset M$ be the locus of stable bundles of type $E_i$, so that each $Y_i$ is the bijective image of the projective space $\P\Ext^1(K(-E_i),\OO(E_i))$.  Then $M$ is the disjoint union of $U$ and the $Y_i$.  
 
The most important step of the proof is to construct the blowdown map $\pi\colon M \to \P\Ext^1(K,\OO)$,
which we now describe.  Every bundle $V\in M$ is either of type $\OO$ or of type $E_i$, so fits in one of the sequences
$$0\lra \OO \lra V\lra K\lra 0$$
and 
$$0\lra \OO(E_i)\lra V\lra K(-E_i)\lra 0.$$
In each case, $V$ has a unique section, so we can canonically consider an exact sequence of the form
$$0\lra \OO \lra V \lra F\lra 0.$$  

If $V$ has type $\OO$, then the sheaf $F$ is isomorphic to $K$.  On the other hand, if $V$ has type $E_i$, then $F$ has a torsion subsheaf $\OO_{E_i}(-1)$ and $F$ is isomorphic to $K(-E_i)\oplus \OO_{E_i}(-1)$ since there are no nontrivial extensions of $K(-E_i)$ by $\OO_{E_i}(-1)$.  In either case, $\Hom(K,F)$ is a $1$-dimensional space, and applying $\Hom(K,-)$ to the sequence gives us an inclusion
$$0\lra \Hom(K,F)\lra \Ext^1(K,\OO)$$
since $\Hom(K,V) = 0$.  This inclusion therefore determines a point in $\P\Ext^1(K,\OO)$, which we denote by $\pi(V)$.  Carrying out this construction in families defines a morphism $\pi \colon M\to \P\Ext^1(K,\OO)$.  

 If $V$ has type $\OO$, it is clear that $\pi(V)$ is precisely (the linear space spanned by) the extension class defining~$V$.  Thus $\pi$ acts on $U$ by the natural inclusion $U \to \P\Ext^1(K,\OO)$.

Suppose that $V$ has type $E_i$.  Then we claim that $\pi(V) = q_i$.  In fact, we have an isomorphism $\Hom(K,F) \cong \Hom(K,\OO_{E_i}(-1))$, so we just need to see that up to this isomorphism, the inclusion $\Hom(K,F)\to \Ext^1(K,\OO)$ is the same map as the canonical inclusion $\Hom(K,\OO_{E_1}(-1))\to \Ext^1(K,\OO)$.  We have the following diagram of short exact sequences: 
$$
\xymatrix{
0 \ar[r] &\OO \ar[r] & V \ar[r] & F \ar[r] & 0 \\
0 \ar[r] & \OO \ar@{=}[u]\ar[r] & \OO(E_i)\ar[u]\ar[r] & \OO_{E_i}(-1)\ar[u]\ar[r] & 0\rlap{.} 
}
$$
Applying $\Hom(K,-)$, we get the commutative square
$$
\xymatrix{
\Hom(K,F) \ar[r] & \Ext^1(K,\OO)\\
\Hom(K,\OO_{E_i}(-1)) \ar[r]\ar[u]^\cong & \Ext^1(K,\OO)\rlap{,}\ar@{=}[u]
}
$$
which shows that the image of $\Hom(K,F)$ in $\Ext^1(K,\OO)$ is the subspace corresponding to $q_i$.

Next we will show that $M$ is smooth and irreducible.  Since $M$ has a component of dimension $n-11$ and every tangent space of $M$ has dimension $n-11$, it suffices to show that $M$ is connected.  Consider any mapping of a smooth curve $C \to \P\Ext^1(K,\OO)$ having a point $p\in C$ mapping to $q_i$.  Then the map $C\sm \{p\} \to M$ extends to a regular map $C\to M$, and the point $p$ maps to a point representing a bundle of some type $E_j$.  By the continuity of the map $\pi$, the only possibility is that $j=i$.  The locus $Y_i$ is itself connected, so we conclude that $M$ is connected.  Therefore, $M$ is smooth.

By the universal property of the blowup, the map $\pi\colon M\to \P\Ext^1(K,\OO)$ factors through the blowup of $\P\Ext^1(K,\OO)$ at $\{q_1,\ldots,q_n\}$ as
$$M \lra \Bl_{q_1,\ldots,q_n} \P\Ext^1(K,\OO) \lra \P\Ext^1(K,\OO).$$
The first map is a bijection between smooth varieties, so it is an isomorphism by Zariski's main theorem.  
\end{proof}

\subsection{Summary of results}
We briefly summarize our description of the moduli spaces $M_{A_t}(2,K,2)$ for $10\leq n \leq 15$, assuming that the SHGH conjecture holds.  In the next section, we will unconditionally describe the moduli space when $n=16$.

\begin{enumerate}
\item If $t > t_\OO = \frac{n}{3}$, then the moduli space $M_{A_t}(2,K,2)$ is empty.

\item If $11\leq n \leq 15$, a new component parameterizing bundles of type $\OO$ arises when $t$ decreases past $t_\OO$.  This component is isomorphic to $\PP^{n-11}$, and the component persists as $t$ decreases to $\sqrt{n}$. If $13\leq n \leq 15$, then as $t$ decreases past $t_{E_1} = \frac{n-2}{3}$, this component is blown up at $n$ points, with the exceptional divisors parameterizing bundles of type $E_i$.

\item For each nontrivial, nonexceptional effective divisor $D$ satisfying $\chi(D) \geq 1$ and $2B\cdot D < B\cdot K$, a new component of $M_{A_t}(2,K,2)$ parameterizing bundles of type $D$ arises when $t$ decreases past $t_D$.  This component is isomorphic to $\P^{-\chi(2D-K)-1}$, and it persists and is unmodified as $t$ decreases to $\sqrt{n}$.  All components of the moduli space are disjoint from each other.
\end{enumerate}

\begin{example}\label{ex-components10}
For $10\leq n\leq 12$, our classification of the possible divisors $D$ in terms of the continued fraction expansion of $\sqrt{n}$ allows us to easily list all the wall-crossings that arise in the following tables: 
\renewcommand{\arraystretch}{1.3}
$$
\begin{array}{ccc}
n=10 && n=11\\
\begin{array}{ccc}
D & t_D &\textrm{New component}\\\hline
57H-18E &  \frac{370}{117} & \P^8\\
2220H-702E & \frac{14050}{4443} & \P^{359}\\
84357H-26676E & \frac{533530}{168717} & \P^{13688}\\ \vdots & \vdots & \vdots \\ &
\end{array} &&
\begin{array}{ccc}
D & t_D &\textrm{New component}\\\hline
\OO & \frac{11}{3} & \P^0\\
30H-9E &  \frac{209}{63} & \P^9\\
627H-189E & \frac{4169}{1257} & \P^{198}\\
12537H-3780E & \frac{83171}{25077} & \P^{3969}\\
\vdots & \vdots & \vdots
\end{array}\\

n=12\\
\begin{array}{ccc}
D & t_D &\textrm{New component}\\\hline
\OO & 4 & \P^1\\
21H-6E &  \frac{52}{15} & \P^{10}\\
312H-90E & \frac{724}{209} & \P^{145}\\
4365H-1260E & \frac{10084}{2911} & \P^{2026}\\
\vdots & \vdots & \vdots
\end{array}\\
\end{array}
$$
\end{example}

\begin{example}\label{ex-components13}
Since we also have the classification of divisors $D$ for $n=13$, we can similarly list all the wall-crossings in this case.  The main additional complication is that $D$ does not have to have equal multiplicities.  For such a $D$, when $t$ decreases past $t_D$, many components will simultaneously arise by permuting the multiplicities of $D$.  We list, in order of decreasing $t_D$,  all of the wall-crossings where $t_D - \sqrt{13} > 10^{-13}$.  The ``type'' indicates the infinite family that the divisor comes from in Theorem~\ref{thm-D13}.
\renewcommand{\arraystretch}{1.3}
$$
\begin{array}{c}
n=13\\
\begin{array}{cccc}
D & \textrm{Type} & t_D &\textrm{New components}\\\hline
\OO & \textrm{I} & \frac{13}{3} & \P^2 \\
E_1 & \textrm{IV} & \frac{11}{3} & \textrm{none; previous $\P^2$ blown up 13 times}\\
15H - 5E_1 - 4E_{2,\ldots,13}& \textrm{V} & \frac{119}{33} & \textrm{13 copies of $\P^{10}$}\\
195H - 54E & \textrm{II} & \frac{1417}{393} &  \P^{119}\\
1962H-545E_1 - 544E_{2,\ldots,13} & \textrm{VI} & \frac{14159}{3927} & \textrm{13 copies of $\P^{1189}$}\\
2142H - 594E & \textrm{I} & \frac{15457}{4287} & \P^{1298}\\
21417H - 5950E_{1,\ldots,12} - 5949E_{13} & \textrm{III} & \frac{1181}{327} & \textrm{13 copies of $\P^{12970}$}\\
2782260H - 771660E & \textrm{I} & \frac{20063173}{5564523} & \P^{1684802}\\
255057H - 70740E & \textrm{II} & \frac{1839253}{510117} & \P^{154451}\\
2548620H-706860E_{1,\ldots,12} - 706859E_{13} & \textrm{IV} & \frac{18378371}{5097243}
& \textrm{13 copies of $\P^{1543321}$}\\ \vdots & \vdots & \vdots & \vdots
\end{array}\\
\end{array}
$$
\end{example}

\subsection{Smaller Euler characteristic}
We have focused entirely on the moduli spaces $M_{A_t}(2,K,2)$ so far in this section, but our analysis makes it easy to prove a qualitative statement about the components of $M_{A_t}(2,K,\chi)$ for any $\chi \leq 2$.  For concreteness, we will restrict our attention to $10\leq n\leq 12$, although with a detailed analysis of the divisors $D$, it would be easy to extend the statement to $10\leq n\leq 15$. 

\begin{theorem}
$($Assume SHGH\,$)$ Suppose $10\leq n\leq 12$ and $\chi$ is an integer with $\chi \leq 2$.  Fix positive integers $k$ and $r$.  There exists an $\epsilon > 0$ such that if $\sqrt{n}< t < \sqrt{n}+\epsilon$, then the moduli space $M_{A_t} (2,K,\chi)$ has at least $k$ irreducible components of dimension at least $r$.
\end{theorem}

\begin{proof}
As $t$ approaches $\sqrt{n}$, the moduli space $M_{A_t}(2,K,2)$ gets components corresponding to the divisors $D_3,D_5,D_7,\ldots$.  We know that the dimensions of these components grow like the denominators in the continued fraction expansion of $\sqrt{n}$ by Remark~\ref{rem-growth}.  Then if $t$ is sufficiently close to $\sqrt{n}$, we can arrange that $M_{A_t}(2,K,2)$ has at least $k$ irreducible components $M_1,\ldots,M_k$ of dimension at least $r$.  Thus the theorem is true for $\chi=2$.  If $\chi<2$, we can additionally arrange that the difference in the dimensions between any two of these components is as large as we want; for concreteness, let us say that any two of these components differ in dimension by more than $2-\chi$.  These components are projective spaces, and in particular they are smooth. 

Recall that if $V$ is a torsion-free sheaf and $p\in X$ is a point where $V$ is locally free, then an \emph{elementary modification} of $V$ at $p$ is a sheaf $V'$ fitting in a sequence
$$0\lra V' \lra V \lra \OO_p\lra 0.$$
If $V$ is $\mu_{A_t}$-stable, then so is $V'$.  Then the locus in $M_{A_t}(2,K,1)$ parameterizing the elementary modifications of sheaves in $M_1$ is irreducible, so lies in an irreducible component $M_1'$ of $M_{A_t}(2,K,1)$.  By the analysis in \cite[Section 3.3]{CoskunHuizengaPathologies}, the dimension of the component $M_1'$ satisfies
$$\dim M_1 + 3 \leq \dim M_1' \leq \dim M_1 + 4.$$
Similarly, if we instead perform $2-\chi$ general elementary modifications to the bundles in $M_1$, then the resulting bundles will lie in an irreducible component $M_1^{(2-\chi)}$ of $M_{A_t}(2,K,\chi)$ whose dimension satisfies
$$\dim M_1 + 3(2-\chi) \leq \dim M_1^{(2-\chi)} \leq \dim M_1 + 4(2-\chi).$$
If we carry out this process for each of the components $M_1,\ldots, M_k$, we obtain a list of components $M_1^{(2-\chi)},\ldots,M_k^{(2-\chi)}$ of $M_{A_t}(2,K,\chi)$.  Our assumption on the dimensions of $M_1,\ldots,M_k$ implies that these components each have distinct dimensions, and they all have dimension at least $r$.
\end{proof}

\begin{remark}
In contrast, if the polarization $A_t$ is fixed but $\chi$ becomes arbitrarily negative, then the moduli spaces $M_{A_t}(2,K,\chi)$ become irreducible by O'Grady's theorem; see \cite{OGrady}.  Thus it is necessary to choose the polarization $A_t$ after fixing the Euler characteristic $\chi$ in the previous theorem.
\end{remark}

\section{Moduli spaces for sixteen or twenty-five points }\label{sec-square}

When $n=16$, the results of the previous section can all be proven independently of the SHGH conjecture.  In this section, we indicate the modifications that need to be made to the arguments to remove this dependency.  We then also discuss the moduli space $M_{A_t}(2,K,4)$ when $n=25$; by similar arguments, these spaces can also be completely described, independently of the SHGH conjecture.  We begin with the following theorem, which summarizes our results in case $n=16$.

\begin{theorem}
Let $n=16$.
\begin{enumerate}
\item For any $t$ with $t_{E_1} < t < t_\OO$, the moduli space $M_{A_t}(2,K,2)$ is isomorphic to $\P^5$
\item For any $t$ with $4 < t < t_{E_1}$, the moduli space $M_{A_t}(2,K,2)$ is isomorphic to the blowup of\, $\P^5$ at sixteen points.  Under the identification $\P^5 \cong \P\Ext^1(K,\OO)$, these sixteen points correspond to the images of the inclusions $\Hom(K,\OO_{E_i}(-1)) \to \Ext^1(K,\OO)$.
\end{enumerate}
\end{theorem}

\begin{proof}
We essentially have to repeat the sequence of arguments in Section~\ref{sec-components}, making modifications whenever the SHGH conjecture was used.  The conjecture was primarily used when appealing to Section~\ref{sec-cohomology} to determine cohomological properties of possible divisors $D$ which could lead to destabilizing objects.  However, when $n=16$, we have the complete unconditional classification of divisors $D$ satisfying $\chi(D)\geq 1$ and $2B\cdot D < B\cdot K$ provided by Theorem~\ref{thm-D16}: the only possible $D$ are $\OO$ and the $E_i$.  For these divisors, the statements in Section~\ref{sec-cohomology} become trivial, so this will be fairly straightforward.

In the proof of Proposition~\ref{prop-Dstable}, if a bundle of type $E_i$ is destabilized, then it is destabilized by a line bundle $K(-D')$ such that $D'-E_i$ is nontrivial effective, $2A_t\cdot D' < A_t\cdot K$, and $\chi(D')\geq 1$.  By Theorem~\ref{thm-D16}, there are no such possible $D'$.

Similar modifications can be made to the first paragraph of the proof of Proposition~\ref{prop-Ostable}, and the rest of the proof of that proposition does not refer to SHGH.

The only portion of Lemma~\ref{lem-ext} that is relevant is part~\eqref{le-3}, which clearly holds without SHGH.

Theorem~\ref{thm-componentD} only discusses components corresponding to nontrivial, nonexceptional divisors $D$ satisfying $\chi(D)\geq 1$ and $2B\cdot D < B\cdot K$; as there are no such divisors, the moduli space does not have any additional components.

The proof of Theorem~\ref{thm-componentO} makes use of the previous results from Section~\ref{sec-components}, but does not make any additional use of SHGH.
\end{proof}

Next we consider the case $n=25$ and the moduli space $M_{A_t}(2,K,4)$.  Note that the maximal Euler characteristic of an effective divisor $D$ satisfying $2B\cdot D < B\cdot K$ is $\chi(D) = 2$.  The argument in Lemma~\ref{lem-free} then shows that the maximal Euler characteristic of an $\mu_{A_t}$-semistable rank $2$ bundle $V$ with $c_1(V) = K$ is $\chi(V) = 4$ and that any $A_t$-semistable sheaf of character $(2,K,4)$ is a vector bundle.

\begin{theorem}
Let $n=25$.  For any $t$ with $5<t<\frac{27}{5}$, the moduli space $M_{A_t}(2,K,4)$ is isomorphic to a disjoint union of $25$ copies of\, $\P^8$.  The copies can be naturally identified with the spaces $\P\Ext^1(K(-D),\OO(D))$, where $D$ is one of the divisors $H-E_i$.
\end{theorem}

\begin{proof}
Recall that by Theorem~\ref{thm-D25},  any effective divisor $D$ satisfying $\chi(D) \geq 2$ and $2B\cdot D < B\cdot K$ is of the form $D = H-E_i$ for some $i$.  If a bundle $V$ of type $D = H-E_i$ is not $A_t$-stable, then it is destabilized by a line bundle $L = K(-D')$.  Here $D'$ must be an effective divisor such that $D'-D$ is nontrivial effective, $2A_t \cdot D' < A_t\cdot K$, and $\chi(D') \geq 2$, as in the proof of Proposition~\ref{prop-Dstable}.  There are no such divisors $D'$. 
 
For these divisors $D$, Lemma~\ref{lem-ext} clearly holds without SHGH.  The proof of Theorem~\ref{thm-componentD} goes through without further modification to complete the result.
\end{proof}

%%%%%%%%%%%%%%%%%%%%%
% References
%%%%%%%%%%%%%%%%%%%%%

\bibliographystyle{plain}
  
\end{document}